\documentclass[a4paper, 11pt, reqno]{amsart}
\usepackage[utf8]{inputenc}
\usepackage{amssymb,amsmath,amsfonts,amsthm,mathtools,graphicx,enumitem,booktabs,multirow,multicol,array,longtable,pdflscape,colortbl,caption}
\usepackage[hidelinks]{hyperref}
\usepackage{tikz}
\usetikzlibrary{cd}

\AtBeginDocument{}
\newcolumntype{M}[1]{>{\centering\arraybackslash}m{#1}}

\newtheorem{theorem}{Theorem}[section]
\newtheorem{corollary}[theorem]{Corollary}
\newtheorem{lemma}[theorem]{Lemma}
\newtheorem{proposition}[theorem]{Proposition}
\newtheorem{question}[theorem]{Question}

\newtheorem{remark}[theorem]{Remark}
\theoremstyle{definition}

\DeclarePairedDelimiter{\parens}{\lparen}{\rparen}
\DeclarePairedDelimiterX{\pres}[2]{\langle}{\rangle}{#1\,\delimsize\vert\,\mathopen{}#2}

\newcommand*{\cont}[1]{\mathrm{cont}\parens{#1}}
\newcommand*{\supp}[1]{\mathrm{supp}\parens{#1}}

\newcommand{\bbN}{\mathbb{N}}
\newcommand{\bfV}{\mathbf{V}}
\newcommand{\bfW}{\mathbf{W}}
\newcommand{\calR}{\mathcal{R}}
\newcommand{\calX}{\mathcal{X}}
\newcommand{\uord}{\mathbf{u}}
\newcommand{\vord}{\mathbf{v}}
\newcommand{\word}{\mathbf{w}}
\newcommand{\pord}{\mathbf{p}}
\newcommand{\qord}{\mathbf{q}}
\newcommand{\rord}{\mathbf{r}}
\newcommand{\sord}{\mathbf{s}}

\newcommand*{\hypo}{{\mathsf{hypo}}}
\newcommand*{\hypon}{{\mathsf{hypo}_n}}
\newcommand*{\sylv}{{\mathsf{sylv}}}
\newcommand*{\sylvn}{{\mathsf{sylv}_n}}
\newcommand*{\sylvh}{{\mathsf{sylv}^{\#}}}
\newcommand*{\sylvhn}{{\mathsf{sylv}^{\#}_n}}
\newcommand*{\baxt}{{\mathsf{baxt}}}
\newcommand*{\baxtn}{{\mathsf{baxt}_n}}
\newcommand*{\rst}{{\mathsf{rSt}}}
\newcommand*{\rstn}{{\mathsf{rSt}_n}}
\newcommand*{\lst}{{\mathsf{lSt}}}
\newcommand*{\lstn}{{\mathsf{lSt}_n}}
\newcommand*{\mst}{{\mathsf{mSt}}}
\newcommand*{\mstn}{{\mathsf{mSt}_n}}
\newcommand*{\jst}{{\mathsf{jSt}}}
\newcommand*{\jstn}{{\mathsf{jSt}_n}}
\newcommand*{\rtg}{{\mathsf{rTg}}}

\newcommand*{\ltg}{{\mathsf{lTg}}}

\newcommand*{\hs}{{\mathsf{hs}}}
\newcommand*{\hsn}{{\mathsf{hs}_n}}
\newcommand*{\ms}{{\mathsf{ms}}}
\newcommand*{\msn}{{\mathsf{ms}_n}}

\newcommand{\VarMd}{\mathbf{M}_\mathrm{2}}
\newcommand{\VarS}{\mathbf{S}}
\newcommand*{\vhypo}{\bfV_\hypo}
\newcommand*{\vsylv}{\bfV_\sylv}
\newcommand*{\vsylvh}{\bfV_\sylvh}
\newcommand*{\vbaxt}{\bfV_\baxt}
\newcommand*{\vrst}{\bfV_\rst}
\newcommand*{\vlst}{\bfV_\lst}
\newcommand*{\vmst}{\bfV_\mst}
\newcommand*{\vjst}{\bfV_\jst}

\newcommand*{\vhs}{\bfV_\hs}

\newcommand{\hlAo}{\hyperlink{id_content_prefix}{\textup{(C\textsubscript{pre})}}}
\newcommand{\hlAt}{\hyperlink{id_content_suffix}{\textup{(C\textsubscript{suf})}}}
\newcommand{\hlDo}{\hyperlink{id_first_occurrence}{\textup{(S\textsubscript{pre})}}}
\newcommand{\hlDt}{\hyperlink{id_last_occurrence}{\textup{(S\textsubscript{suf})}}}
\newcommand{\hlEo}{\hyperlink{id_simple_support_prefix}{\textup{(S\textsubscript{1,pre})}}}
\newcommand{\hlEt}{\hyperlink{id_simple_support_suffix}{\textup{(S\textsubscript{1,suf})}}}
\newcommand{\hlF}{\hyperlink{id_simple_word}{\textup{(Rst\textsubscript{1})}}}
\newcommand{\hlB}{\hyperlink{id_subsequences}{\textup{(Sub\textsubscript{2})}}}
\newcommand{\hlC}{\hyperlink{id_simple_content}{\textup{(Rst\textsubscript{1,v})}}}

\newcommand{\hlId}{\hyperlink{identity_form}{\uord \approx \vord}}

\newcommand{\hrLo}{\hyperref[idlst]{\textup{(L\textsubscript{1})}}}
\newcommand{\hrRo}{\hyperref[idrst]{\textup{(R\textsubscript{1})}}}
\newcommand{\hrLt}{\hyperref[idsylvh]{\textup{(L\textsubscript{2})}}}
\newcommand{\hrRt}{\hyperref[idsylv]{\textup{(R\textsubscript{2})}}}
\newcommand{\hrMd}{\hyperref[idM2]{\textup{(M\textsubscript{2})}}}
\newcommand{\hrMt}{\hyperref[idM3]{\textup{(M\textsubscript{3})}}}
\newcommand{\hrMf}{\hyperref[idM4]{\textup{(M\textsubscript{4})}}}
\newcommand{\hrOot}{\hyperref[idO12]{\textup{(O\textsubscript{1,2})}}}
\newcommand{\hrEot}{\hyperref[idT12]{\textup{(E\textsubscript{1,2})}}}
\newcommand{\hrOto}{\hyperref[idO21]{\textup{(O\textsubscript{2,1})}}}
\newcommand{\hrEto}{\hyperref[idT21]{\textup{(E\textsubscript{2,1})}}}

\newcommand{\hrLato}{\hyperref[fig:lattice_with_sylv_stal]{$\mathbb{L}_1$}}
\newcommand{\hrLatd}{\hyperref[fig:lattice_with_sylv_stal_hypo]{$\mathbb{L}_2$}}
\newcommand{\hrLatt}{\hyperref[fig:lattice_with_sylv_stal_hypo_and_Md]{$\mathbb{L}_3$}}

\title[Lattices of varieties of plactic-like monoids]{Lattices of varieties of plactic-like monoids}

\date{January 26, 2024}

\author[Thomas Aird]{Thomas Aird}
\address[Thomas Aird]{Department of Mathematics, The University of Manchester, Alan Turing Building, Oxford Rd, Manchester, M13 9PL, UK.}
\email{thomas.aird@manchester.ac.uk}
\thanks{The first author's work was supported by the London Mathematical Society, the Heilbronn Institute for Mathematical Research, and NOVA University Lisbon.}

\author[Duarte Ribeiro]{Duarte Ribeiro}
\address[Duarte Ribeiro]{Center for Mathematics and Applications (NOVA Math), FCT NOVA, 2829-516 Caparica, Portugal.}
\email{dc.ribeiro@campus.fct.unl.pt}
\thanks{The second author's work was supported by National Funds through the FCT -- Funda\c{c}\~{a}o para a Ci\^{e}ncia e a Tecnologia, I.P., under the scope of the projects PTDC/MAT-PUR/31174/2017, UIDB/04621/2020 and UIDP/04621/2020 (Center for Computational and Stochastic Mathematics), and UIDB/00297/2020 (\url{https://doi.org/10.54499/UIDB/00297/2020}) and UIDP/00297/2020 (\url{https://doi.org/10.54499/UIDP/00297/2020}) (Center for Mathematics and Applications).}

\begin{document}

\begin{abstract}
    We study the equational theories and bases of meets and joins of several varieties of plactic-like monoids. Using those results, we construct sublattices of the lattice of varieties of monoids, generated by said varieties. We calculate the axiomatic ranks of their elements, obtain plactic-like congruences whose corresponding factor monoids generate varieties in the lattice, and determine which varieties are joins of the variety of commutative monoids and a finitely generated variety. We also show that the hyposylvester and metasylvester monoids generate the same variety as the sylvester monoid.
\end{abstract}

\maketitle

\section{Introduction}

Plactic-like monoids, whose elements can be uniquely identified with combinatorial objects, have been the focus of intense study in recent years, in particular with regard to their equational theories. The initial motivation to study this was to obtain natural examples of finitely-generated polynomial-growth semigroups that did not satisfy non-trivial identities, as an alternative to the constructions given in \cite{shneerson_identities}. The plactic monoid, whose elements can be viewed as semistandard Young tableaux, was defined by Lascoux and Schützenberger \cite{LS1978}, and found to have important applications in several different subjects, such as representation theory \cite{green2006polynomial}, symmetric functions \cite{macdonald_symmetric} and crystal bases \cite{bump_crystalbases}. Its finite-rank versions were candidates for the previously mentioned problem. However, after some initial results on the rank 2 and 3 cases \cite{jaszunska_chinese,izhakian_cloaktic,kubat_identities}, Johnson and Kambites \cite{johnson_kambites_tropical_plactic} gave faithful representations of said monoids in monoids of upper triangular matrices over the tropical semiring, which are known to satisfy non-trivial identities \cite{izhakian_identities,okninski_identities,taylor2017upper}. On the other hand, Cain et al.~\cite{ckkmo_placticidentity} gave a lower bound for the length of identities satisfied by plactic monoids of finite rank, dependant on said rank, thus showing that the infinite-rank case does not satisfy any non-trivial identity. Furthermore, the first author \cite{aird2023semigroup} has shown that plactic monoids of different ranks generate different varieties, by obtaining a new set of identities satisfied by these monoids.

The study of equational theories has extended to other plactic-like monoids, that arise in the context of combinatorial Hopf algebras whose bases are indexed by combinatorial objects. Of note, the hypoplactic monoid $\hypo$ \cite{novelli_hypoplactic}, the sylvester and \#-sylvester monoids $\sylv$ and $\sylvh$ \cite{hivert_sylvester}, the Baxter monoid $\baxt$ \cite{giraudo_baxter} and the left and right stalactic monoids $\ltg$ and $\rtg$ \cite{hnt_stalactic} have been studied by several authors (including the second author, in joint work with Cain and Malheiro) and different means \cite{cain_johnson_kambites_malheiro_representations_2022,cm_identities,cain_malheiro_ribeiro_sylvester_baxter_2023,cain_malheiro_ribeiro_hypoplactic_2022,han2021preprint}. It was shown that, within each of these classes, monoids of rank greater than or equal to 2 generate the same variety, and full characterisations of equational theories, finite bases and axiomatic ranks were obtained for each case.

Following on the authors' work on factor monoids of the free monoid by meets and joins of left and right stalactic congruences \cite{aird_ribeiro_join_meet_stalactic_2023}, where it was shown that the varieties generated by these monoids are, respectively, the varietal join and meet of the varieties generated by the left and right stalactic monoids, we propose the study of the sublattice of the lattice of varieties of monoids generated by varieties of these plactic-like monoids, so as to understand the underlying connections between these monoids, as well as to motivate the study of congruences given by meets and joins of plactic-like congruences.

The paper is organised as follows: Necessary background is given in Section~\ref{section:background}. We then study three sublattices of the lattice of varieties of monoids in the following three sections. In each section, we first study the varietal meets and joins arising from the generators, with regards to their equational theories, finite bases, if they are generated by factor monoids of the free monoid by meets and joins of plactic-like congruences, and whether they are the varietal join of the variety of commutative monoids and a finitely generated variety, or if they are not contained in any such varietal join. Then, we prove the correctness of the lattice given at the start of the section, and finally, we obtain the axiomatic ranks of the varieties in the lattice. In Section~\ref{section:lattice_L1}, we study the sublattice generated by the varieties generated, respectively, by the \#-sylvester, sylvester, left stalactic and right stalactic monoids. Then, in Section~\ref{section:lattice_L2}, we add a generator, the variety generated by the hypoplactic monoid. Finally, in Section~\ref{section:lattice_L3}, we add another generator, the variety defined by the identity $xzxyty \approx xzyxty$. We then show, in Section~\ref{section:hypo_metasylvester}, that the hyposylvester and metasylvester monoids, recently introduced by Novelli and Thibon \cite{novelli_hypo_meta}, generate the same variety as the sylvester monoid, and conclude our paper in Section~\ref{section:conclusions} with some corollaries and an open question, as well as the collected results in Table~\ref{table:results}.

\section{Background} \label{section:background}

\subsection{Words} \label{subsection:words}

Let $\bbN = \{1,2,\dots\}$ denote the set of natural numbers, without zero. For $n \in \bbN$, we denote by $[n]$ the set $\{1 < \cdots < n\}$. 

For any non-empty finite or countable set $\calX$, we denote by $\calX^*$ the \emph{free monoid} generated by $\calX$, that is, the set of all words over $\calX$ under concatenation. We refer to $\calX$ as an \emph{alphabet}, and its elements as \emph{letters}. The empty word is denoted by $\varepsilon$. For a word $\word \in \calX^*$, we denote its \emph{length} by $|\word|$ and, for each $x \in \calX$, we denote the number of occurrences of $x$ in $\word$ by $|\word|_x$. We say $x$ is a \emph{simple letter} of $\word$ if $|\word|_x = 1$. The subset of $\calX$ of letters $x$ such that $|\word|_x \geq 1$ is called the \emph{support} of $\word$, denoted by $\supp{\word}$, and the function from $\calX$ to $\bbN_0$ given by $x \mapsto |\word|_x$ is called the \emph{content} of $\word$, denoted by $\cont{\word}$. 

For words $\uord,\vord \in \calX^*$ we say that $\uord$ is a \emph{factor} of $\vord$ if there exist $\vord_1,\vord_2 \in \calX^*$ such that $\vord = \vord_1 \uord \vord_2$, and that $\uord$ is a \emph{subsequence} of $\vord$ if there exist $u_1, \dots, u_k \in \calX$ and $\vord_1, \dots, \vord_{k+1} \in \calX^*$ such that $\uord = u_1 \cdots u_k$ and $\vord = \vord_1 u_1 \vord_2 \cdots \vord_k u_k \vord_{k+1}$. 

Given a congruence $\rho$ on $\calX^*$ and a word $\word$, we denote its congruence class by $[\word]_\rho$.

\subsection{Identities and varieties} \label{subsection:identities_and_varieties}

For a general background on universal algebra, see \cite{Bergman_universal_algebra,bs_universal_algebra}. For recent results on varieties of semigroups and monoids, see \cite{gusev_lee_vernikov_survey,lee_book}. The following background is given in the context of monoids.

An \emph{identity} over an alphabet of variables $\calX$ is a formal equality $\uord \approx \vord$, where $\uord, \vord \in \calX^*$. A variable $x$ is said to \emph{occur} in an identity if $x$ occurs in at least one of the sides of the identity, and is said to be a \emph{simple variable} in the identity if it is a simple letter of each side of the identity. An identity $\uord \approx \vord$ is \emph{non-trivial} if $\uord \neq \vord$, and \emph{balanced} if $\cont{\uord} = \cont{\vord}$. Two words with the same content must have the same length, hence we say the \emph{length} of a balanced identity is the length of its left or right-hand side. Two identities are \emph{equivalent} if one can be obtained from the other by renaming variables or swapping both sides of the identities.  

A monoid $M$ \emph{satisfies} the identity $\uord \approx \vord$ if for every morphism $\psi\colon \calX^* \to M$, we have $\psi(\uord) = \psi(\vord)$. We refer to these morphisms as \emph{evaluations}. Notice that if $M$ satisfies $\uord \approx \vord$, then it satisfies any other identity obtained by removing all occurrences of a variable in $\uord \approx \vord$. If an evaluation $\psi$ is such that $\psi(\uord) \neq \psi(\vord)$, we say $\psi$ \emph{falsifies} the identity. A word $\uord$ is an \emph{isoterm} for $M$ if no non-trivial identity of the form $\uord \approx \vord$ is satisfied by $M$. The \emph{identity-checking problem} of $M$ is the combinatorial decision problem of deciding whether an identity is satisfied or not by $M$. Its time complexity is measured in terms of the size of the input, that is, the sum of the lengths of each side of the formal equality. 

The set of identities that are satisfied by all monoids in a class $\mathbf{K}$ is called its \emph{equational theory}, and the class of monoids that satisfy all identities in a set of identities $\Sigma$ is called its \emph{variety}. By Birkhoff's $HSP$-theorem, a class of monoids is a variety if and only if it is closed under taking homomorphic images, submonoids and direct products. We say an identity is satisfied by a variety if it lies in its equational theory, and a word is an isoterm for a variety if it is an isoterm for at least one monoid in said variety. A \emph{subvariety} is a subclass of a variety that is itself a variety. A variety is \emph{generated} by a monoid $M$ if it is the smallest variety containing $M$, and is denoted by $\bfV_M$. The identity-checking problem of $\bfV_M$ is equivalent to that of $M$. A variety is \emph{finitely generated} if it is generated by a finite monoid.

The variety of commutative monoids, denoted by $\mathbf{COM}$, is defined by the set of all balanced identities. A variety is \emph{overcommutative} if it contains the variety of all commutative monoids. As such, all identities satisfied by an overcommutative variety are balanced.

A congruence $\equiv$ on a monoid $M$ is \emph{fully invariant} if $a \mathrel{\equiv} b$ implies $f(a) \mathrel{\equiv} f(b)$, for every $a,b \in M$ and every endomorphism $f$ of $M$. Equational theories over an alphabet $\calX$ are fully invariant congruences on $\calX^*$ (see, for example, \cite[II\S{14}]{bs_universal_algebra}), and any variety is generated by the factor monoid of $\calX^*$ by its equational theory. An identity $\uord \approx \vord$ is a \emph{consequence} of a set of identities $\Sigma$ if, for $1 \leq i \leq k$, there exist words $\pord_i, \qord_i, \rord_i, \sord_i, \word_i, \word_{k+1} \in \mathcal{X}^*$ and endomorphisms $\psi_i$ of $\calX^*$ such that $\uord = \word_1$, $\vord = \word_{k+1}$ and 
\[
     \word_i = \rord_i \psi_i\parens{\pord_i} \sord_i \quad \text{and} \quad\word_{i+1} = \rord_i \psi_i\parens{\qord_i} \sord_i,
\]
where $\pord_i \approx \qord_i$ or $\qord_i \approx \pord_i$ are in $\Sigma$. Notice that any consequence of a set of balanced identities must also be balanced. An \emph{equational basis} of a variety is a subset of its equational theory whose set of consequences is the equational theory itself. We denote a variety with equational basis $\Sigma$ by $\bfV_\Sigma$. A variety is \emph{finitely based} if it admits a finite equational basis. The \emph{axiomatic rank} of a variety is the least natural number such that the variety admits a basis where the number of distinct variables occurring in each identity of the basis does not exceed said number. A finitely based variety is \emph{hereditarily finitely based} if it only has finitely based subvarieties, and a monoid is hereditarily finitely based if it generates a hereditarily finitely based variety.

The class of all varieties of monoids forms a lattice under set-theoretical inclusion, denoted by $\mathbb{MON}$. Given two varieties $\bfV$ and $\bfW$, the equational theory of the varietal meet $\bfV \wedge \bfW$ is the join of the equational theories of $\bfV$ and $\bfW$, and the equational theory of the varietal join $\bfV \vee \bfW$ is the meet of their equational theories. Since equational theories are fully invariant congruences, the meet of two equational theories is their intersection. Furthermore, if both $\bfV$ and $\bfW$ are finitely based, then $\bfV \wedge \bfW$ is also finitely based, and admits the union of the respective finite bases for $\bfV$ and $\bfW$ as a finite basis. Notice that the join of two finitely generated varieties is also finitely generated, since the direct product of the finite generators is not only finite, but also generates the join.  

\subsection{Properties defining equational theories} \label{subsection:properties_equational_theories}

Let $\uord \approx \vord$ be an identity over the alphabet of variables $\calX$, such that $\uord$ and $\vord$ share the same support and simple variables. We say $\uord \approx \vord$ satisfies the property
\begin{enumerate}
    \hypertarget{id_content_prefix}{\item[(C\textsubscript{pre})]} if, for any variable $x \in \supp{\uord \approx \vord}$, the shortest prefix of $\uord$ where $x$ occurs has the same \emph{content} as the shortest prefix of $\vord$ where $x$ occurs;
    \hypertarget{id_content_suffix}{\item[(C\textsubscript{suf})]} if, for any variable $x \in \supp{\uord \approx \vord}$, the shortest suffix of $\uord$ where $x$ occurs has the same \emph{content} as the shortest suffix of $\vord$ where $x$ occurs;
    \hypertarget{id_subsequences}{\item[(Sub\textsubscript{2})]} if $\uord$ and $\vord$ share the same subsequences of length 2;
    \hypertarget{id_simple_content}{\item[(Rst\textsubscript{1,v})]} if, for any variable $x \in \supp{\uord \approx \vord}$, the word obtained from $\uord$ by restricting it to $x$ and its simple variables is the same as that obtained from $\vord$;
    \hypertarget{id_first_occurrence}{\item[(S\textsubscript{pre})]} if, for any variable $x \in \supp{\uord \approx \vord}$, the shortest prefix of $\uord$ where $x$ occurs has the same \emph{support} as the shortest prefix of $\vord$ where $x$ occurs;
    \hypertarget{id_last_occurrence}{\item[(S\textsubscript{suf})]} if, for any variable $x \in \supp{\uord \approx \vord}$, the shortest suffix of $\uord$ where $x$ occurs has the same \emph{support} as the shortest suffix of $\vord$ where $x$ occurs;
    \hypertarget{id_simple_support_prefix}{\item[(S\textsubscript{1,pre})]} if, for any \emph{simple} variable $x \in \supp{\uord \approx \vord}$, the shortest prefix of $\uord$ where $x$ occurs has the same \emph{support} as the shortest prefix of $\vord$ where $x$ occurs;
    \hypertarget{id_simple_support_suffix}{\item[(S\textsubscript{1,suf})]} if, for any \emph{simple} variable $x \in \supp{\uord \approx \vord}$, the shortest suffix of $\uord$ where $x$ occurs has the same \emph{support} as the shortest suffix of $\vord$ where $x$ occurs;
    \hypertarget{id_simple_word}{\item[(Rst\textsubscript{1})]} if the word obtained from $\uord$ by restricting it to its simple variables is the same as that obtained from $\vord$.
\end{enumerate}

Clearly, some of these properties are stronger than others. For example, an identity satisfying property \hlB{} also satisfies property \hlEo{}, as with the case of $xzxytx \approx xzyxtx$, but the converse does not necessarily hold, as with the case of $x^2y \approx xyx$. The following diagram illustrates the connections between these properties:

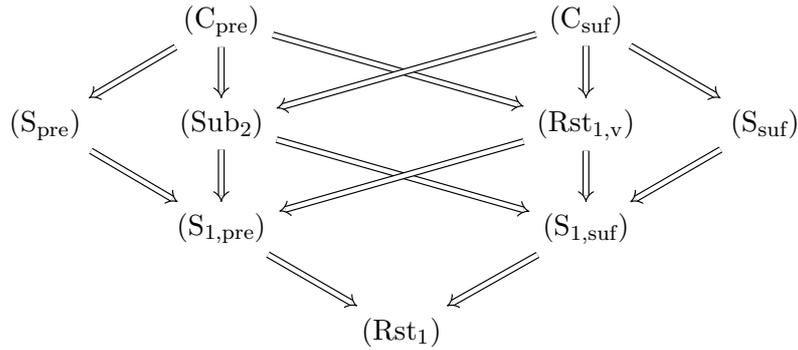
\[
\begin{tikzcd}
    & \hlAo{} \arrow[dl, Rightarrow] \arrow[d, Rightarrow] \arrow[drr, Rightarrow] & & \hlAt{} \arrow[dr, Rightarrow] \arrow[d, Rightarrow]  \arrow[dll, Rightarrow] & \\
    \hlDo{} \arrow[dr, Rightarrow] & \hlB{} \arrow[d, Rightarrow] \arrow[drr, Rightarrow] & & \hlC{} \arrow[d, Rightarrow] \arrow[dll, Rightarrow] & \hlDt{} \arrow[dl, Rightarrow] \\
    & \hlEo{} \arrow[dr, Rightarrow] & & \hlEt{} \arrow[dl, Rightarrow] & \\
    & & \hlF{} & & \\
\end{tikzcd}
\]

These properties define equational theories (see, for example, \cite{gerhard_petrich_orthogroups} and \cite{sapir_finitely_based_monoids}), some of which are of finitely-generated varieties: Consider the four-element monoids $J^1$ and $\overleftarrow{J^1}$ given by the monoid presentations $\pres{a,b}{ab=0, ba=a,b^2=b}$ and $\pres{a,b}{ab=a, ba=0,b^2=b}$, respectively. On one hand, it was shown by Edmunds \cite{edmunds_order_four} that $J^1$ is finitely based by the set of identities 
\[
    \{x^2 \approx x^3, yx^2 \approx xyx, x^2 y^2 \approx y^2 x^2\}.
\]
On the other hand, Gusev and Vernikov \cite[Proposition~4.2]{gusev_vernikov_18_chain} showed that the equational theory defined by the dual set of these identities is described by \hlEo{}. From these two results, we obtain the following:

\begin{lemma} \label{lemma:J1_dual_identities}
    The equational theory of $\bfV_{J^1}$ is the set of identities that satisfy \hlEt{}, and the equational theory of $\bfV_{\overleftarrow{J^1}}$ is the set of identities that satisfy \hlEo{}.
\end{lemma}

Notice that if we only consider balanced identities, then property \hlC{} is equivalent to the following: for any \emph{simple} variable $x \in \supp{\uord \approx \vord}$, the shortest prefix of $\uord$ where $x$ occurs has the same \emph{content} as the shortest prefix of $\vord$ where $x$ occurs. As we mostly deal with balanced identities, we will predominantly work with this equivalent definition. On the other hand, property \hlB{} is equivalent to the following property, denoted by $\mathcal{P}_{1,2}$ in \cite{sapir_finitely_based_monoids}: for any variables $x,y \in \supp{\uord \approx \vord}$, the first occurrence of $x$ occurs before the last occurrence of $y$ in $\uord$ if and only if it does so in $\vord$. 

\begin{remark} \label{remark:polynomial_time}
    It is clear that checking if a balanced identity satisfies any of the properties \hlAo{}--\hlF{}, or any combination of them, can be done in polynomial time.
\end{remark}

\subsection{Plactic-like monoids} \label{subsection:plactic_like_monoids}

`Plactic-like' monoids are an informal class of monoids whose elements can be bijectively identified with certain combinatorial objects. Its namesake is the plactic monoid \cite{LS1978}, also known as the monoid of Young tableaux. In this work, we will approach these monoids from a syntactic perspective, as plactic-like monoids can be defined as factors of the free monoid, over a finite or countable alphabet, by their respective plactic-like congruences. To be precise, for a plactic-like congruence $\equiv$, the infinite rank plactic-like monoid is the factor monoid $\bbN^*/{\equiv}$, and the plactic-like monoid of finite rank $n$ is the factor monoid $[n]^*/{\equiv}$.

\hypertarget{sylvh_sylv_definition}{The \#-sylvester and sylvester congruences \cite{hivert_sylvester,giraudo_baxter} are generated, respectively, by the relations
\begin{align*}
    \calR_\sylvh = \{(b \uord ac,b \uord ca) \colon a < b \leq c, \uord \in \bbN^*\} &\quad \text{and} \\ 
    \calR_\sylv = \{(ca \uord b,ac \uord b) \colon a \leq b < c, \uord \in \bbN^*\}&.
\end{align*}
The \#-sylvester monoids of countable rank and finite rank $n$ are denoted, respectively, by $\sylvh$ and $\sylvhn$, while the sylvester monoids of countable rank and finite rank $n$ are denoted, respectively, by $\sylv$ and $\sylvn$.}

For a word $\word \in \bbN^*$ and $a,b \in \supp{\word}$ such that $a<b$, we say $\word$ has an $a$-$b$ \emph{left precedence} (of index $k$) if $\word = \word_1 b \word_2$, where $|\word_1|_a = k$ and $|\word_1|_c = 0$, for all $a < c \leq b$. Two words are $\equiv_\sylvh$-congruent if and only if they share the same content and left precedences \cite[Proposition~2.10]{cain_malheiro_ribeiro_sylvester_baxter_2023}. 

Similarly, we say $\word$ has a $b$-$a$ \emph{right precedence} (of index $k$) if $\word = \word_1 a \word_2$, where $|\word_2|_b = k$ and $|\word_2|_c = 0$, for all $a \leq c < b$. Two words are $\equiv_\sylv$-congruent if and only if they share the same content and right precedences \cite[Proposition~2.7]{cain_malheiro_ribeiro_sylvester_baxter_2023}.

The following was proven independently in \cite[Theorems~4.1~and~4.2]{cain_malheiro_ribeiro_sylvester_baxter_2023} and \cite[Lemma~3.3~and~Theorem~3.4]{han2021preprint}:

\begin{theorem} \label{theorem:vsylvh_vsylv_identities}
    The equational theory of $\vsylvh$ is the set of balanced identities that satisfy the property \hlAo{}, and the equational theory of $\vsylv$ is the set of balanced identities that satisfy the property \hlAt{}.
\end{theorem}

\begin{proposition}[{\cite[Proposition~6.6~(3)]{cain_johnson_kambites_malheiro_representations_2022}}] \label{prop:vsylvh_vsylv_varietal_join}
    Neither $\vsylvh$ nor $\vsylv$ are contained in the join of $\mathbf{COM}$ and any finitely generated variety. 
\end{proposition}
    
The following was proven independently in \cite[Theorems~4.16~and~4.17]{cain_malheiro_ribeiro_sylvester_baxter_2023}, \cite[Theorem~6.7~and~Remark~6.8]{cain_johnson_kambites_malheiro_representations_2022} and \cite[Theorem~3.4]{han2021preprint}:

\begin{theorem} \label{theorem:vsylvh_vsylv_finite_basis}
    The variety $\vsylvh$ admits a finite equational basis consisting of the identity 
    \begin{align*}
        xzytxy &\approx xzytyx, \tag{L\textsubscript{2}} \label{idsylvh}
    \end{align*}
    and the variety $\vsylv$ admits a finite equational basis consisting of the identity 
    \begin{align*}
        xyzxty &\approx yxzxty. \tag{R\textsubscript{2}} \label{idsylv}
    \end{align*}
\end{theorem}

\begin{corollary}[{\cite[Corollary~4.22]{cain_malheiro_ribeiro_sylvester_baxter_2023}}] \label{cor:vsylvh_vsylv_axiomatic_rank}
    The axiomatic rank of $\vsylvh$ and $\vsylv$ is $4$.
\end{corollary}

\hypertarget{baxt_definition}{The Baxter congruence \cite{giraudo_baxter} is the meet of the sylvester and \#-sylvester congruences, and it is generated by the relation
\begin{align*}
	\mathcal{R}_{\baxt} =& \left\{ (c \uord da \vord b, c \uord ad \vord b): a \leq b < c \leq d, \uord,\vord \in \bbN^* \right\}\\
	& \cup \left\{ (b \uord da \vord c, b \uord ad \vord c): a < b \leq c < d, \uord,\vord \in \bbN^* \right\}.
\end{align*}
The Baxter monoids of countable rank and finite rank $n$ are denoted, respectively, by $\baxt$ and $\baxtn$.}

Two words are $\equiv_\baxt$-congruent if and only if they share the same content and left and right precedences \cite[Corollary~2.11]{cain_malheiro_ribeiro_sylvester_baxter_2023}.

As a consequence of \cite[Proposition~3.7]{giraudo_baxter}, we have that $\vsylvh \vee \vsylv = \vbaxt$.

The following was proven independently in \cite[Theorem~4.3]{cain_malheiro_ribeiro_sylvester_baxter_2023} and \cite[Lemma~3.3 and Theorem~3.8]{han2021preprint}:

\begin{theorem} \label{theorem:vbaxt_identities}
    The equational theory of $\vbaxt$ is the set of balanced identities that satisfy the properties \hlAo{} and \hlAt{}.
\end{theorem}

\begin{proposition}[{\cite[Proposition~6.10~(4)]{cain_johnson_kambites_malheiro_representations_2022}}] \label{prop:vbaxt_varietal_join}
    The variety $\vbaxt$ is not contained in the join of $\mathbf{COM}$ and any finitely generated variety. 
\end{proposition}

The following was proven independently in \cite[Theorem~4.18]{cain_malheiro_ribeiro_sylvester_baxter_2023}, \cite[Theorem~6.11]{cain_johnson_kambites_malheiro_representations_2022} and \cite[Theorem~3.8]{han2021preprint}:

\begin{theorem} \label{theorem:vbaxt_finite_basis}
    The variety $\vbaxt$ admits a finite equational basis consisting of the identities
    \begin{align*}
        xzytxyrxsy &\approx xzytyxrxsy, \tag{O\textsubscript{2,2}} \label{idbaxtO}\\
        xzytxyrysx &\approx xzytyxrysx. \tag{T\textsubscript{2,2}} \label{idbaxtT}
    \end{align*}
\end{theorem}

\begin{corollary}[{\cite[Corollary~4.24]{cain_malheiro_ribeiro_sylvester_baxter_2023}}] \label{cor:vbaxt_axiomatic_rank}
    The axiomatic rank of $\vbaxt$ is $6$.
\end{corollary}

\hypertarget{hypo_definition}{The hypoplactic congruence \cite{novelli_hypoplactic} is the join of the sylvester and \#-sylvester congruences and, as such, generated by the relations $\calR_\sylvh \cup \calR_\sylv$. The hypoplactic monoids of countable rank and finite rank $n$ are denoted, respectively, by $\hypo$ and $\hypon$.}

For a word $\word \in \bbN^*$ and $a,b \in \supp{\word}$ such that $a<b$ and there is no $c \in \supp{\word}$ such that $a<c<b$, we say $\word$ has a $b$-$a$ inversion if it admits the subsequence $ba$. Two words are $\equiv_\hypo$-congruent if and only if they share the same content and inversions \cite[Subsection~4.2]{novelli_hypoplactic}.

By the following results, we can see that $\vhypo$ is not the meet of $\vsylvh$ and $\vsylv$:

\begin{theorem}[{\cite[Theorem~4.1]{cain_malheiro_ribeiro_hypoplactic_2022}}] \label{theorem:vhypo_identities}
    The equational theory of $\vhypo$ is the set of balanced identities that satisfy the property \hlB{}.
\end{theorem}

The variety $\mathbf{J}_2$, generated by the five-element monoid of all order-preserving and extensive transformations of the three-element chain, is defined by the set of identities that satisfy property \hlB{} \cite[Theorem~2]{volkov_reflexive_relations}.

\begin{corollary}[{\cite[Corollary~4.4]{cain_malheiro_ribeiro_hypoplactic_2022}}] \label{cor:vhypo_varietal_join}
    The variety $\vhypo$ is the varietal join of $\mathbf{COM}$ and $\mathbf{J}_2$.
\end{corollary}

\begin{theorem}[{\cite[Theorem~4.8]{cain_malheiro_ribeiro_hypoplactic_2022}}] \label{theorem:vhypo_finite_basis}
    The variety $\vhypo$ admits a finite equational basis consisting of the identities \hrLt{}, \hrRt{} and
    \begin{align*}
        xyxzx &\approx x^2yzx. \tag{M\textsubscript{3}} \label{idM3} 
    \end{align*}
\end{theorem}

\begin{remark}
    The identity \hrMt{} is different from its corresponding identity given in \cite[Theorem~4.8]{cain_malheiro_ribeiro_hypoplactic_2022}, however they are consequences of one another, thus we can replace one with another and still obtain a basis for $\vhypo$.
\end{remark}

\begin{corollary}[{\cite[Corollary~4.12]{cain_malheiro_ribeiro_hypoplactic_2022}}] \label{cor:vhypo_axiomatic_rank}
    The axiomatic rank of $\vhypo$ is $4$.
\end{corollary}

\hypertarget{lst_rst_definition}{The left-stalactic and right-stalactic congruences \cite{hnt_stalactic,aird_ribeiro_join_meet_stalactic_2023} are generated, respectively, by the relations
\begin{align*}
    \calR_\lst = \{(a \uord ab, a \uord ba) \colon a,b \in \bbN, \uord \in \bbN^*\} &\quad \text{and} \\ 
    \calR_\rst = \{(ab \uord a, ba \uord a) \colon a,b \in \bbN, \uord \in \bbN^*\}&.
\end{align*}
The left-stalactic monoids of countable rank and finite rank $n$ are denoted, respectively, by $\lst$ and $\lstn$, while the right-stalactic monoids of countable rank and finite rank $n$ are denoted, respectively, by $\rst$ and $\rstn$.}

Two words are $\equiv_\lst$-congruent if and only if they share the same content and order of first occurrences of symbols \cite[Subsection~3.7]{hnt_stalactic}. Similarly, two words are $\equiv_\rst$-congruent if and only if they share the same content and order of last occurrences of symbols.

The following was proven independently in \cite[Corollary~4.6]{cain_johnson_kambites_malheiro_representations_2022} and \cite[Lemma~2.1~and~Theorem~2.3]{han2021preprint}:

\begin{corollary} \label{cor:vlst_vrst_identities}
    The equational theory of $\vlst$ is the set of balanced identities that satisfy the property \hlDo{}, and the equational theory of $\vrst$ is the set of balanced identities that satisfy the property \hlDt{}.
\end{corollary}

The variety of left (resp. right) regular band monoids $\mathbf{LRB}$ (resp. $\mathbf{RRB}$) is defined by the set of identities that satisfy property \hlDo{} (resp. \hlDt{}) \cite{gerhard_petrich_orthogroups}, and is generated by the left (resp. right) flip-flop monoid, a two-element left (resp. right) zero semigroup with an identity adjoined \cite[Proposition~7.3.2]{rhodes_steinberg_q_theory}.

\begin{corollary}[{\cite[Corollary~4.6~(2)]{cain_johnson_kambites_malheiro_representations_2022}}] \label{cor:vlst_vrst_varietal_join}
    The variety $\vlst$ is the varietal join of $\mathbf{COM}$ and $\mathbf{LRB}$, and the variety $\vrst$ is the varietal join of $\mathbf{COM}$ and $\mathbf{RRB}$.
\end{corollary}

The following was proven independently in \cite[Corollary~4.6]{cain_johnson_kambites_malheiro_representations_2022} and \cite[Theorem~2.3]{han2021preprint}:

\begin{corollary} \label{cor:vlst_vrst_finite_basis}
    The variety $\vlst$ admits a finite equational basis consisting of the identity 
    \begin{align*}
        xyx &\approx x^2y, \tag{L\textsubscript{1}} \label{idlst}
    \end{align*}
    and the variety $\vrst$ admits a finite equational basis consisting of the identity
    \begin{align*}
        xyx &\approx yx^2. \tag{R\textsubscript{1}} \label{idrst}
    \end{align*}
\end{corollary}

\begin{corollary} \label{cor:vlst_vrst_axiomatic_rank}
    The axiomatic rank of $\vlst$ and $\vrst$ is $2$.
\end{corollary}
\begin{proof}
    Follows from Corollary~\ref{cor:vlst_vrst_finite_basis}, and $\vlst$ and $\vrst$ being overcommutative.
\end{proof}

\hypertarget{mst_definition}{The meet-stalactic congruence \cite{aird_ribeiro_join_meet_stalactic_2023} is the meet of the left and right-stalactic congruences, and it is generated by the relation
\begin{align*}
    \calR_\mst :=& \{ (b \uord ba \vord b,b \uord ab \vord b)\colon a,b \in \bbN, \uord,\vord \in \bbN^* \} \\
    & \cup \{ (a \uord ab \vord b,a \uord ba \vord b)\colon a,b \in \bbN, \uord,\vord \in \bbN^* \}.
\end{align*}
The meet-stalactic monoids of countable rank and finite rank $n$ are denoted, respectively, by $\mst$ and $\mstn$.}

Two words are $\equiv_\mst$-congruent if and only if they share the same content and order of first and last occurrences of symbols.

As a consequence of \cite[Proposition~7.3]{aird_ribeiro_join_meet_stalactic_2023}, we have that $\vrst \vee \vlst = \vmst$.

\begin{corollary}[{\cite[Corollary~7.5]{aird_ribeiro_join_meet_stalactic_2023}}] \label{cor:vmst_identities}
    The equational theory of $\vmst$ is the set of balanced identities that satisfy the properties \hlDo{} and \hlDt{}.
\end{corollary}

The variety of regular band monoids $\mathbf{RB}$ is the varietal join of $\mathbf{LRB}$ and $\mathbf{RRB}$. Thus, it is defined by the set of identities satisfying both properties \hlDo{} and \hlDt{}, and generated by the direct product of the left and right flip-flop monoids.

\begin{corollary}[{\cite[Corollary~7.13]{aird_ribeiro_join_meet_stalactic_2023}}] \label{cor:vmst_varietal_join}
    The variety $\vmst$ is the varietal join of $\mathbf{COM}$ and $\mathbf{RB}$.
\end{corollary}

\begin{corollary}[{\cite[Corollary~7.7]{aird_ribeiro_join_meet_stalactic_2023}}] \label{cor:vmst_finite_basis}
    The variety $\vmst$ admits a finite equational basis consisting of the identities \hrMt{} and
    \begin{align*}
        xzxyty &\approx xzyxty. \tag{M\textsubscript{2}} \label{idM2}
    \end{align*}
\end{corollary}

\begin{corollary}[{\cite[Corollary~7.10]{aird_ribeiro_join_meet_stalactic_2023}}] \label{cor:vmst_axiomatic_rank}
    The axiomatic rank of $\vmst$ is $4$.
\end{corollary}

\hypertarget{jst_definition}{The join-stalactic congruence \cite{aird_ribeiro_join_meet_stalactic_2023} is the join of the left-stalactic and right-stalactic congruences and, as such, generated by the relations $\calR_\lst \cup \calR_\rst$. The join-stalactic monoids of countable rank and finite rank $n$ are denoted, respectively, by $\jst$ and $\jstn$.}

Two words are $\equiv_\jst$-congruent if and only if they share the same content and order of simple letters \cite[Proposition~3.2]{aird_ribeiro_join_meet_stalactic_2023}.

Another consequence of \cite[Proposition~7.3]{aird_ribeiro_join_meet_stalactic_2023} is that $\vrst \wedge \vlst = \vjst$.

\begin{corollary}[{\cite[Corollary~7.15]{aird_ribeiro_join_meet_stalactic_2023}}] \label{cor:vjst_identities}
    The equational theory of $\vjst$ is the set of balanced identities that satisfy the property \hlF{}.
\end{corollary}

Consider the monoid $S(\{ab\})$, the Rees factor monoid over the ideal of $\bbN^*$ consisting of all words that are not factors of $ab$, which is a finite monoid with zero. The variety $\bfV_{S(\{ab\})}$ is defined by the set of identities that satisfy property \hlF{} (see \cite[Table~1]{sapir_finitely_based_monoids}).

\begin{corollary} \label{cor:vjst_varietal_join}
    The variety $\vjst$ is the varietal join of $\mathbf{COM}$ and $\bfV_{S(\{ab\})}$.
\end{corollary}

\begin{corollary}[{\cite[Corollary~7.17]{aird_ribeiro_join_meet_stalactic_2023}}] \label{cor:vjst_finite_basis}
    The variety $\vjst$ admits a finite equational basis consisting of the identities \hrLo{} and \hrRo{}.
\end{corollary}

\begin{corollary}[{\cite[Corollary~7.18]{aird_ribeiro_join_meet_stalactic_2023}}] \label{cor:vjst_axiomatic_rank}
    The axiomatic rank of $\vjst$ is $2$.
\end{corollary}

\begin{proposition}[{\cite[Proposition~7.19]{aird_ribeiro_join_meet_stalactic_2023}}] \label{prop:vjst_no_proper_overcommutative_varieties}
    $\vjst$ is the unique cover of $\mathbf{COM}$ in the lattice of all varieties of monoids. 
\end{proposition} 

\section{Sublattice of $\mathbb{MON}$ generated by $\vsylvh$, $\vsylv$, $\vlst$ and $\vrst$.} \label{section:lattice_L1}

In this section, we construct \hrLato{}, the sublattice of $\mathbb{MON}$ generated by the varieties $\vsylvh$, $\vsylv$, $\vlst$ and $\vrst$. We start by studying the equational theories and bases of all possible varietal meets and joins obtained from the generators, and then from the obtained elements together with the generators, then we show that no more varieties occur in the lattice and the covers are well-defined, and finally we obtain their axiomatic ranks.

\begin{figure}[h] 
    \centering
    \begin{tikzpicture}
        \draw[very thick] (0,1) -- (-2,0) -- (-4,-1) -- (-4,-3) -- (-4,-8) -- (-2,-9) -- (0,-10);
        \draw[very thick] (0,1) -- (2,0) -- (4,-1) -- (4,-3) -- (4,-8) -- (2,-9) -- (0,-10);
        \draw[very thick] (-2,0) -- (0,-1) -- (2,0);
        \draw[very thick] (2,-9) -- (0,-8) -- (-2,-9);
        \draw[very thick] (-2,-4.5) -- (-4,-3) -- (0,-1) -- (4,-3) -- (-2,-4.5) -- (0,-8);
        \draw[very thick] (0,-1) -- (2,-4.5) -- (-4,-6) -- (0,-8) -- (4,-6) -- (2,-4.5);
        \filldraw[black] (0,1) circle (2pt) node[anchor=south]{$\vbaxt$};
        \filldraw[black] (-2,0) circle (2pt);
        \filldraw[black] (2,0) circle (2pt);
        \filldraw[black] (-4,-1) circle (3pt) node[anchor=east]{$\vsylvh$};
        \filldraw[black] (0,-1) circle (2pt);
        \filldraw[black] (4,-1) circle (3pt) node[anchor=west]{$\vsylv$};
        \filldraw[black] (-4,-3) circle (2pt);
        \filldraw[black] (4,-3) circle (2pt);
        \filldraw[black] (-2,-4.5) circle (2pt) node[above=2pt]{$\VarS$};
        \filldraw[black] (2,-4.5) circle (2pt) node[right=5pt]{$\vmst$};
        \filldraw[black] (-4,-6) circle (2pt);
        \filldraw[black] (4,-6) circle (2pt);
        \filldraw[black] (-4,-8) circle (3pt) node[anchor=east]{$\vlst$};
        \filldraw[black] (4,-8) circle (3pt) node[anchor=west]{$\vrst$};
        \filldraw[black] (0,-8) circle (2pt);
        \filldraw[black] (-2,-9) circle (2pt);
        \filldraw[black] (2,-9) circle (2pt);
        \filldraw[black] (0,-10) circle (2pt) node[anchor=north]{$\vjst$};
    \end{tikzpicture}
    \caption{Sublattice $\mathbb{L}_1$ of $\mathbb{MON}$ generated by the varieties generated by the \#-sylvester, sylvester, and left and right stalactic monoids.}
    \label{fig:lattice_with_sylv_stal}
\end{figure}
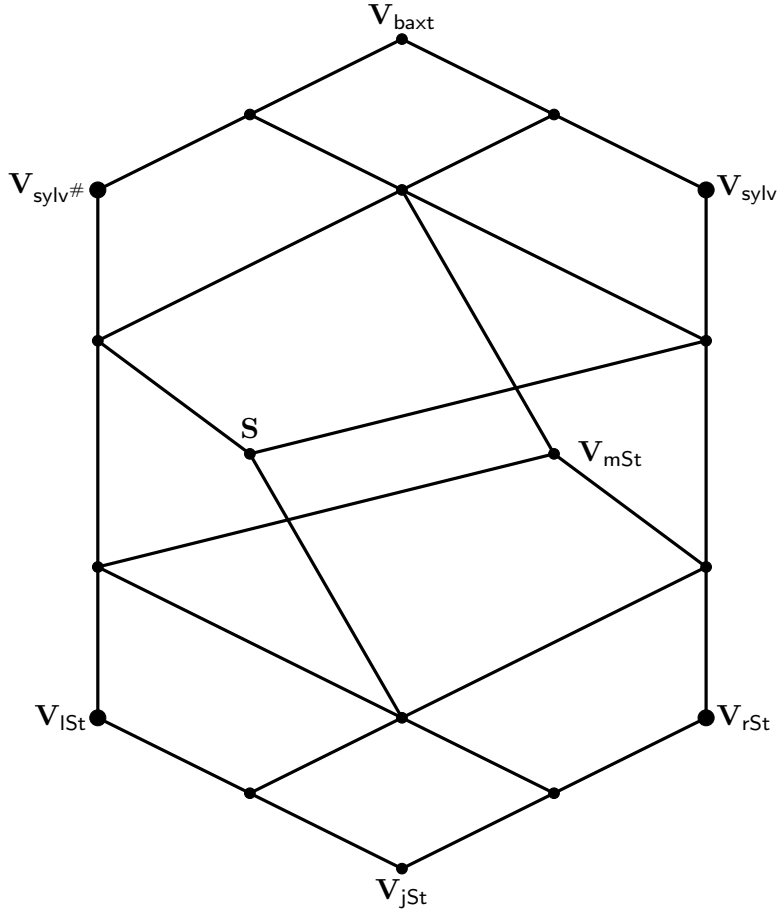

\begin{theorem} \label{theorem:lattice_1}
    The Hasse diagram of \hrLato{} is given in Figure~\ref{fig:lattice_with_sylv_stal}.
\end{theorem}

To simplify the notation, we denote $\vsylvh \wedge \vsylv$ by $\VarS$. This variety has been studied by Sapir \cite{sapir_finitely_based_monoids}, in a different context. The next result follows from Theorem~\ref{theorem:vsylvh_vsylv_finite_basis}:

\begin{corollary} \label{cor:VarS_finite_basis}
    The variety $\VarS$ admits a finite equational basis consisting of the identities \hrLt{} and \hrRt{}.
\end{corollary}

\begin{proposition}[{\cite[Proposition~6.2]{sapir_finitely_based_monoids}}] \label{prop:VarS_identities}
    The equational theory of $\VarS$ is the set of balanced identities that satisfy the properties \hlB{} and \hlC{}.
\end{proposition}

In order to construct the lattice, we begin by determining which varieties are incomparable, as these have non-trivial meets and joins:

\begin{lemma} \label{lemma:incomparabilities_containments}
	The following statements hold:
    \begin{enumerate}[label=(\roman*)]
    	\item $\vlst$, $\vrst$ and $\VarS$ are pairwise incomparable and contain $\vjst$;
    	\item $\vmst$, $\vsylvh$ and $\vsylv$ are pairwise incomparable and contained in $\vbaxt$;
        \item $\vlst$ is incomparable with $\vsylv$ and contained in $\vmst$ and $\vsylvh$;
        \item $\vrst$ is incomparable with $\vsylvh$ and contained in $\vmst$ and $\vsylv$;
        \item $\VarS$ is incomparable with $\vmst$ and contained in $\vsylvh$ and $\vsylv$.
    \end{enumerate} 
\end{lemma}
\begin{proof}
    It is clear that that $\vlst$ and $\vrst$ (resp. $\vsylvh$ and $\vsylv$) are incomparable, since they are defined by dual identities. Furthermore, $\vlst$ (resp. $\vrst$) is contained in $\vsylvh$ (resp. $\vsylv$), by Theorem~\ref{theorem:vsylvh_vsylv_identities} and Corollary~\ref{cor:vlst_vrst_identities}. By Corollary~\ref{cor:vjst_identities} and Proposition~\ref{prop:VarS_identities}, $\vjst$ is contained in $\VarS$, and by Theorem~\ref{theorem:vbaxt_identities} and Corollary~\ref{cor:vmst_identities}, $\vmst$ is contained in $\vbaxt$.
    
    By Corollary~\ref{cor:vmst_finite_basis} and Proposition~\ref{prop:VarS_identities}, the identity \hrMt{} is satisfied by $\vjst, \vlst, \vrst$ and $\vmst$, but not by $\VarS, \vsylv, \vsylvh$ or $\vbaxt$. 
    Furthermore, by Theorem~\ref{theorem:vsylvh_vsylv_finite_basis}, \hrRt{} is satisfied by $\vsylv$ and $\VarS$, and \hrLt{} is satisfied by $\vsylvh$ and $\VarS$. However, by Corollary~\ref{cor:vlst_vrst_identities}, \hrRt{} is not satisfied by $\vlst$ or $\vmst$, and \hrLt{} is not satisfied by $\vrst$ or $\vmst$.
\end{proof}

\subsection{Varietal meets} \label{subsection:L1_meets}

\begin{proposition} \label{prop:vlst_meet_vsylv_vrst_meet_vsylvh_finite_basis}
    The variety $\vlst \wedge \vsylv$ admits a finite equational basis consisting of the identities \hrLo{} and
    \begin{align*}
        x^2y^2 \approx y^2x^2, \tag{M\textsubscript{4}} \label{idM4}
    \end{align*}
    and the variety $\vrst \wedge \vsylvh$ admits a finite equational basis consisting of the identities \hrRo{} and \hrMf{}.
\end{proposition}
\begin{proof}
    By Theorem~\ref{theorem:vsylvh_vsylv_finite_basis} and Corollary~\ref{cor:vlst_vrst_finite_basis}, $\vlst \wedge \vsylv$ admits an equational basis consisting of \hrLo{} and \hrRt{}. Note that
    \[ x^2y^2 \approx xyxy \approx yx^2y \approx y^2x^2\]
    is satisfied by $\vlst \wedge \vsylv$, as the first and third identities are consequences of \hrLo{}, and the second is a consequence of \hrRt{}. Moreover,
    \[xyzxty \approx x^2y^2zt \approx y^2x^2zt \approx yxzxty\]
    is in the equational theory given by \hrLo{} and \hrMf{}, as the first and third identities are consequences of \hrLo{}, and the second is a consequence of \hrMf{}. 
    
    \par Hence, \hrLo{} and \hrMf{} form an equational basis for $\vlst \wedge \vsylv$. A dual argument works for the case of $\vrst \wedge \vsylvh$.
\end{proof}

\begin{proposition} \label{prop:vlst_meet_vsylv_vrst_meet_vsylvh_identities}
    The equational theory of $\vlst \wedge \vsylv$ is the set of balanced identities that satisfy the property \hlEo{}, and the equational theory of $\vrst \wedge \vsylvh$ is the set of balanced identities that satisfy the property \hlEt{}.
\end{proposition}
\begin{proof}
    Clearly, the identities \hrLo{} and \hrMf{}, which form an equational basis for $\vlst \wedge \vsylv$ by Proposition~\ref{prop:vlst_meet_vsylv_vrst_meet_vsylvh_finite_basis}, satisfy property \hlEo{}. As such, all identities in the equational theory of $\vlst \wedge \vsylv$ satisfy this property as well.

    Let $\uord \approx \vord$ be a balanced identity satisfying property \hlEo{}. By Theorem~\ref{theorem:vsylvh_vsylv_finite_basis} and Corollary~\ref{cor:vlst_vrst_identities}, we have that $\vlst \wedge \vsylv$ satisfies the identities \hrLt{}, \hrRt{} and \hrMd{}. As such, it follows from \cite[Proposition~11.2]{lee_book} that 
    \[
        \vlst \wedge \vsylv \wedge \bfV_{\{\uord \approx \vord\}} = \vlst \wedge \vsylv \wedge \bfV_\Sigma
    \]
    for some set $\Sigma$ of non-trivial identities of the form
    \[
        x^{e_0} \prod_{i=1}^{r} (h_i x^{e_i}) \approx x^{f_0} \prod_{i=1}^{r} (h_i x^{f_i}),
    \]
    which are balanced and satisfy property \hlEo{}. Notice that, in \cite[Proposition~11.2]{lee_book}, two sets of identities, of different forms, are given. The identities in the second set are consequences of the identity \hrRt{}, so we do not need to consider them in this case. Clearly, identities in $\Sigma$ are consequences of the identity \hrLo{} and so, by Proposition~\ref{prop:vlst_meet_vsylv_vrst_meet_vsylvh_finite_basis}, they are satisfied by $\vlst \wedge \vsylv$. Hence, 
    \[
        \vlst \wedge \vsylv \wedge \bfV_{\{\uord \approx \vord\}} = \vlst \wedge \vsylv
    \]
    and, therefore, $\uord \approx \vord$ is satisfied by $\vlst \wedge \vsylv$. A dual argument works for the case of $\vrst \wedge \vsylvh$.
\end{proof}

\begin{proposition} \label{prop:vlst_meet_vsylv_vrst_meet_vsylvh_generators}
    The variety $\vlst \wedge \vsylv$ is generated by the factor monoid $\bbN^*/{(\equiv_\lst \vee \equiv_\sylv)}$, and the variety $\vrst \wedge \vsylvh$ is generated by the factor monoid $\bbN^*/{(\equiv_\rst \vee \equiv_\sylvh)}$.
\end{proposition}
\begin{proof}
    It is clear that $\bbN^*/{(\equiv_\lst \vee \equiv_\sylv)} \in \vlst \wedge \vsylv$, since it is a homomorphic image of both $\lst$ and $\sylv$. On the other hand, let $\uord \approx \vord$ be a non-trivial identity satisfied by $\bbN^*/{(\equiv_\lst \vee \equiv_\sylv)}$. Clearly, it is a balanced identity, and if no variable occurring in it is simple, it trivially satisfies \hlEo{}. 
    
    Suppose, in order to obtain a contradiction, that there exist $x,y \in \supp{\uord \approx \vord}$ such that $x$ is simple and $y$ occurs before $x$ in $\uord$, but not in $\vord$. Consider the evaluation $\psi$ such that $x \mapsto [1]_{(\equiv_\lst \vee \equiv_\sylv)}$, $y \mapsto [2]_{(\equiv_\lst \vee \equiv_\sylv)}$ and $z \mapsto [\varepsilon]_{(\equiv_\lst \vee \equiv_\sylv)}$, for any other variable $z \in \supp{\uord \approx \vord}$. Then, the only word in $\psi(\vord)$ is $12^{|\vord|_y}$, since no reordering of $12^{|\vord|_y}$ starts with the letter $1$ or has a $2$-$1$ right precedence of index $|\vord|_y$, and as such, $12^{|\vord|_y}$ forms a singleton class in both $\lst$ and $\sylv$. Hence, $\psi$ falsifies the identity, and we obtain a contradiction. 
    
    Hence, all identities in the equational theory of $\bbN^*/{(\equiv_\lst \vee \equiv_\sylv)}$ satisfy \hlEo{}, and the result follows. A dual argument works for the case of $\vrst \wedge \vsylvh$.
\end{proof}

\begin{corollary} \label{cor:vlst_meet_vsylv_vrst_meet_vsylvh_varietal_join}
    The variety $\vlst \wedge \vsylv$ is the varietal join of $\mathbf{COM}$ and $\bfV_{\overleftarrow{J^1}}$, and the variety $\vrst \wedge \vsylvh$ is the varietal join of $\mathbf{COM}$ and $\bfV_{J^1}$.
\end{corollary}
\begin{proof}
    Follows from Lemma~\ref{lemma:J1_dual_identities} and Proposition~\ref{prop:vlst_meet_vsylv_vrst_meet_vsylvh_identities}.
\end{proof}

\begin{corollary} \label{cor:vmst_meet_VarS_finite_basis}
    The variety $\vmst \wedge \VarS$ admits a finite equational basis consisting of the identities \hrLt{}, \hrRt{}, \hrMd{} and \hrMt{}.
\end{corollary}
\begin{proof}
    Follows from Corollaries~\ref{cor:vmst_finite_basis} and \ref{cor:VarS_finite_basis}.
\end{proof}

\begin{proposition} \label{prop:vmst_meet_VarS_identities}
    The equational theory of $\vmst \wedge \VarS$ is the set of balanced identities that satisfy the properties \hlEo{} and \hlEt{}.
\end{proposition}
\begin{proof}
    Clearly, the identities \hrLt{}, \hrRt{}, \hrMd{} and \hrMt{}, which form an equational basis for $\vmst \wedge \VarS$ by Corollary~\ref{cor:vmst_meet_VarS_finite_basis}, satisfy both properties \hlEo{} and \hlEt{}. As such, all identities in the equational theory of $\vmst \wedge \VarS$ satisfy these properties as well.

    Let $\uord \approx \vord$ be a balanced identity satisfying properties \hlEo{} and \hlEt{}. It follows from \cite[Proposition~11.2]{lee_book} that 
    \[
        \vmst \wedge \VarS \wedge \bfV_{\{\uord \approx \vord\}} = \vmst \wedge \VarS \wedge \bfV_\Sigma
    \]
    for some set $\Sigma$ of non-trivial identities of the form
    \[
        x^{e_0} \prod_{i=1}^{r} (h_i x^{e_i}) \approx x^{f_0} \prod_{i=1}^{r} (h_i x^{f_i}),
    \]
    which are balanced and satisfy properties \hlEo{} and \hlEt{}. Clearly, the identities in $\Sigma$ are consequences of the identity \hrMt{} and so, by Corollary~\ref{cor:vmst_meet_VarS_finite_basis}, they are satisfied by $\vmst \wedge \VarS$. Hence, 
    \[  
        \vmst \wedge \VarS \wedge \bfV_{\{\uord \approx \vord\}} = \vmst \wedge \VarS
    \]
    and, therefore, $\uord \approx \vord$ is satisfied by $\vmst \wedge \VarS$. 
\end{proof}

\begin{corollary} \label{cor:vmst_meet_VarS_varietal_join}
    The variety $\vmst \wedge \VarS$ is the varietal join of $\mathbf{COM}$ and $\bfV_{\overleftarrow{J^1} \times J^1}$.
\end{corollary}
\begin{proof}
    Follows from Lemma~\ref{lemma:J1_dual_identities} and Proposition~\ref{prop:vmst_meet_VarS_identities}.
\end{proof}

\begin{proposition} \label{prop:vmst_meet_VarS_generators}
    The variety $\vmst \wedge \VarS$ is generated by the factor monoid $\bbN^*/{(\equiv_\hypo \vee \equiv_\mst)}$.
\end{proposition}
\begin{proof}
    It is clear that $\bbN^*/{(\equiv_\hypo \vee \equiv_\mst)} \in \vmst \wedge \VarS$, since it is a homomorphic image of $\mst$, $\sylvh$ and $\sylv$. 

    Let $\uord \approx \vord$ be a non-trivial identity satisfied by $\bbN^*/{(\equiv_\hypo \vee \equiv_\mst)}$.  Suppose now that $x,y \in \supp{\uord \approx \vord}$ are such that $x$ is simple and $y$ occurs after $x$ in $\uord$, but not in $\vord$. For $\psi \colon x \mapsto [2]_{(\equiv_\hypo \vee \equiv_\mst)}$, $y \mapsto [1]_{(\equiv_\hypo \vee \equiv_\mst)}$ and $z \mapsto [\varepsilon]_{(\equiv_\hypo \vee \equiv_\mst)}$, for any other variable $z \in \supp{\uord \approx \vord}$, since no reordering of $12^{|\vord|_y}$ starts with the letter $1$, all reorderings have a $2$-$1$ inversion, thus the only word in $\psi(\vord)$ is $12^{|\vord|_y}$. Therefore, $\psi$ falsifies the identity, hence $\uord \approx \vord$ satisfies \hlEt{}. By dual reasoning, this identity satisfies \hlEo{}, and the result follows.
\end{proof}

\begin{corollary} \label{cor:vmst_meet_vsylvh_vmst_meet_vsylv_finite_basis}
    The variety $\vmst \wedge \vsylvh$ admits a finite equational basis consisting of the identities \hrLt{}, \hrMd{} and \hrMt{}, and the variety $\vmst \wedge \vsylv$ admits a finite equational basis consisting of the identities \hrRt{}, \hrMd{} and \hrMt{}.
\end{corollary}
\begin{proof}
    Follows from Theorem~\ref{theorem:vsylvh_vsylv_finite_basis} and Corollary~\ref{cor:vmst_finite_basis}.
\end{proof}

\begin{proposition} \label{prop:vmst_meet_vsylvh_vmst_meet_vsylv_identities}
    The equational theory of $\vmst \wedge \vsylvh$ is the set of balanced identities that satisfy the properties \hlDo{} and \hlEt{}, and the equational theory of $\vmst \wedge \vsylv$ is the set of balanced identities that satisfy the properties \hlDt{} and \hlEo{}.
\end{proposition}
\begin{proof}
    Clearly, the identities \hrLt{}, \hrMd{} and \hrMt{}, which form an equational basis for $\vmst \wedge \vsylvh$ by Proposition~\ref{prop:vlst_meet_vsylv_vrst_meet_vsylvh_finite_basis}, satisfy properties \hlDo{} and \hlEt{}. As such, all identities in the equational theory of $\vmst \wedge \vsylvh$ satisfy these properties as well.

    Let $\uord \approx \vord$ be a balanced identity satisfying properties \hlDo{} and \hlEt{}. As such, it follows from \cite[Proposition~11.2]{lee_book} that 
    \[
        \vmst \wedge \vsylvh \wedge \bfV_{\{\uord \approx \vord\}} = \vmst \wedge \vsylvh \wedge \bfV_\Sigma
    \]
    for some set $\Sigma$ of non-trivial identities of the form
    \[
        x^{e_0} \prod_{i=1}^{r} (h_i x^{e_i}) \approx x^{f_0} \prod_{i=1}^{r} (h_i x^{f_i}),
    \]
    which are balanced and satisfy properties \hlDo{} and \hlEt{}. Notice that, in this case, the second set of identities given in \cite[Proposition~11.2]{lee_book} do not satisfy property \hlDo{}. Clearly, identities in $\Sigma$ are consequences of the identities \hrLt{} and \hrMt{} and so, by Corollary~\ref{cor:vmst_meet_vsylvh_vmst_meet_vsylv_finite_basis}, they are satisfied by $\vmst \wedge \vsylvh$. Hence, 
    \[
        \vmst \wedge \vsylvh \wedge \bfV_{\{\uord \approx \vord\}} = \vmst \wedge \vsylvh
    \]
    and, therefore, $\uord \approx \vord$ is satisfied by $\vmst \wedge \vsylvh$. A dual argument works for the case of $\vmst \wedge \vsylv$.
\end{proof}

\begin{corollary} \label{cor:vmst_meet_vsylvh_vmst_meet_vsylv_varietal_join}
    The variety $\vmst \wedge \vsylvh$ is the varietal join of $\mathbf{COM}$ and $\mathbf{LRB} \vee \bfV_{J^1}$, and the variety $\vmst \wedge \vsylv$ is the varietal join of $\mathbf{COM}$ and $\mathbf{RRB} \vee \bfV_{\overleftarrow{J^1}}$.
\end{corollary}
\begin{proof}
    Follows from Lemma~\ref{lemma:J1_dual_identities}, Corollaries~\ref{cor:vlst_vrst_identities} and \ref{cor:vlst_vrst_varietal_join}, and Proposition~\ref{prop:vmst_meet_vsylvh_vmst_meet_vsylv_identities}.
\end{proof}

\begin{proposition} \label{prop:vmst_meet_vsylvh_vmst_meet_vsylv_generators}
    The variety $\vmst \wedge \vsylvh$ is generated by the factor monoid $\bbN^*/{(\equiv_\mst \vee \equiv_\sylvh)}$, and the variety $\vmst \wedge \vsylv$ is generated by the factor monoid $\bbN^*/{(\equiv_\mst \vee \equiv_\sylv)}$.
\end{proposition}
\begin{proof}
    It is clear that $\bbN^*/{(\equiv_\mst \vee \equiv_\sylvh)} \in \vmst \wedge \vsylvh$, since it is a homomorphic image of both $\mst$ and $\sylvh$. 

    Let $\uord \approx \vord$ be a non-trivial identity satisfied by $\bbN^*/{(\equiv_\mst \vee \equiv_\sylvh)}$. By a dual reasoning to that given in the proof of Proposition~\ref{prop:vlst_meet_vsylv_vrst_meet_vsylvh_generators}, this identity satisfies \hlEt{}.
        
    Suppose now that $x,y \in \supp{\uord \approx \vord}$ are such that $y$ occurs before the first occurrence of $x$ in $\uord$, but not in $\vord$. For $\psi \colon x \mapsto [1]_{(\equiv_\mst \vee \equiv_\sylvh)}$, $y \mapsto [2]_{(\equiv_\mst \vee \equiv_\sylvh)}$ and $z \mapsto [\varepsilon]_{(\equiv_\mst \vee \equiv_\sylvh)}$, for any other variable $z \in \supp{\uord \approx \vord}$, we have that words in $\psi(\vord)$ start with the letter $1$, but words in $\psi(\uord)$ start with $2$, thus they do not share the same left precedences. Therefore, $\psi$ falsifies the identity, hence $\uord \approx \vord$ satisfies \hlDo{}, and the result follows. A dual argument works for the case of $\vmst \wedge \vsylv$.
\end{proof}

\subsection{Varietal joins} \label{subsection:L1_joins}

\begin{corollary} \label{cor:vrst_join_vsylvh_vlst_join_vsylv_identities} 
    The equational theory of $\vrst \vee \vsylvh$ is the set of balanced identities that satisfy the properties \hlAo{} and \hlDt{}, and the equational theory of $\vlst \vee \vsylv$ is the set of balanced identities that satisfy the properties \hlAt{} and \hlDo{}.
\end{corollary}
\begin{proof}
    Follows from Theorem~\ref{theorem:vsylvh_vsylv_identities} and Corollary~\ref{cor:vlst_vrst_identities}.
\end{proof}

\begin{corollary} \label{cor:vrst_join_vsylvh_vlst_join_vsylv_generators}
    The variety $\vrst \vee \vsylvh$ is generated by the factor monoid $\bbN^*/{(\equiv_\rst \wedge \equiv_\sylvh)}$, and the variety $\vlst \vee \vsylv$ is generated by the factor monoid $\bbN^*/{(\equiv_\lst \wedge \equiv_\sylv)}$.
\end{corollary}

\begin{proposition} \label{prop:vrst_join_vsylvh_vlst_join_vsylv_finite_basis}
    The variety $\vrst \vee \vsylvh$ admits a finite equational basis consisting of the identities 
    \begin{align*}
        xzytxyry &\approx xzytyxry, \tag{O\textsubscript{2,1}} \label{idO21}\\
        xzytxyrx &\approx xzytyxrx, \tag{E\textsubscript{2,1}} \label{idT21}
    \end{align*}
    and the variety $\vlst \vee \vsylv$ admits a finite equational basis consisting of the identities 
    \begin{align*}
        xzxytxry &\approx xzyxtxry, \tag{O\textsubscript{1,2}} \label{idO12} \\
        xzxytyrx &\approx xzyxtyrx. \tag{E\textsubscript{1,2}} \label{idT12}
    \end{align*}
\end{proposition}
\begin{proof}
    Clearly, the identities \hrOto{} and \hrEto{} satisfy both properties \hlAo{} and \hlDt{}. So, as these properties define the equational theory of $\vrst \vee \vsylvh$ by Corollary~\ref{cor:vrst_join_vsylvh_vlst_join_vsylv_identities}, all consequences of these identities satisfy these properties as well. 

    We now show that any balanced non-trivial identity satisfying \hlAo{} and \hlDt{} must be a consequence of \hrOto{} and \hrEto{}. The proof will be by induction, in the following sense: We order identities by the length of the common suffix of both sides of the identity, with the induction being on the length of the prefix up to the common suffix. 

    The base case for the induction is the identities of the form
    \[ xyxy \word \approx xyyx \word, \]
    where $x,y \in \calX$ and $\word \in \calX^+$. Notice that $x$ or $y$ must occur in $\word$ by \hlDt{}, hence the identity is a consequence of either \hrOto{} and \hrEto{}, depending on whether $x$ or $y$ occur in $\word$.
        
    Now, let $\uord \approx \vord$ be a balanced non-trivial identity satisfying \hlAo{} and \hlDt{}, of length $n \geq 5$, such that the common suffix $\word \in \calX^*$ of $\uord$ and $\vord$ is such that $|\word| \leq n-4$. Since $\uord \approx \vord$ is a non-trivial identity, we must have 
    \[\uord = \uord' x \word \quad \text{and} \quad \vord = \vord' \word,\]
    for some $x \in \calX$ and $\uord',\vord' \in \calX^+$ such that $x$ is not the last variable of $\vord'$. Furthermore, $\uord$ and $\vord$ share the same content, hence $x$ occurs in $\vord'$. Therefore, we can write 
    \[ \vord' = \vord_1' x \vord_2', \]
    for some words $\vord_1' \in \calX^*$, $\vord_2' \in (\calX\setminus \{x\})^+$.

    Notice that $\supp{x \vord_2'} \subseteq \supp{\vord_1'}$ by \hlAo{}. If $x$ does not occur in $\word$, then $\supp{\vord_2'} \subseteq \supp{\word}$ by \hlDt{}. Thus, the identity $\vord_1' x \vord_2' \word \approx \vord_1'\vord_2' x \word$ is a consequence of \hrOto{} or \hrEto{}, depending on whether $x$ or all the variables of $\vord_2'$ occur in $\word$, and where $x$ and variables of $\vord_2'$ occur in $\vord_1'$. 
    
    As previously mentioned, any consequence of \hrOto{} and \hrEto{} must also satisfy properties \hlAo{} and \hlDt{}. As such, $\vord \approx \vord_1'\vord_2' x \word$ satisfies said properties. By the definition of $\uord \approx \vord$, we can then conclude that the identity $\uord \approx \vord_1'\vord_2' x \word$ is a balanced identity of length $n \geq 5$, satisfying \hlAo{} and \hlDt{}, with a common suffix of length $|\word|+1$. As such, by the induction hypothesis, it is a consequence of \hrOto{} and \hrEto{}, and so is $\uord \approx \vord$. A dual argument works for the case of $\vlst \vee \vsylv$.
\end{proof}

\hypertarget{identity_form}{In the following, we only give the sketches of proofs for results on characterizations of equational bases, as they follow the same reasoning as the one given in the previous proof. For the induction steps, we assume identities $\uord \approx \vord$ are such that
\[ \uord = \uord' x \word \quad \text{and} \quad \vord = \vord_1' x \vord_2' \word, \]
for some $x \in \calX$ and $\uord' \in \calX^+$, $\vord_1' \in \calX^*$, $\vord_2' \in (\calX\setminus \{x\})^+$, where $\word \in \calX^*$ is the common suffix.}

\begin{corollary}
\label{cor:vmst_join_VarS_identities}
    The equational theory of $\vmst \vee \VarS$ is the set of balanced identities that satisfy the properties \hlB{}, \hlC{}, \hlDo{} and \hlDt{}.
\end{corollary}
\begin{proof}
    Follows from Corollary~\ref{cor:vmst_identities} and Proposition~\ref{prop:VarS_identities}.
\end{proof}

\begin{proposition} \label{prop:vmst_join_VarS_finite_basis}
    The variety $\vmst \vee \VarS$ admits a finite equational basis consisting of the identities \hrOot{}, \hrEot{}, \hrOto{} and \hrEto{}.
\end{proposition}
\begin{proof}
    Clearly, the identities \hrOot{}, \hrEot{}, \hrOto{} and \hrEto{} satisfy properties \hlDo{}, \hlDt{}, \hlB{} and \hlC{}, which define the equational theory of $\vmst \vee \VarS$ by Corollary~\ref{cor:vmst_join_VarS_identities}. Now, we prove by induction that any identity satisfying these properties is a consequence of said identities.

    The base case for the induction is the identities of the form
    \[ xyx \word \approx xxy \word. \]
    Notice that $x$ and $y$ must occur in $\word$ by \hlB{} and \hlC{}, respectively. Thus, the identity is a consequence of either \hrOot{} and \hrEot{}, depending on where $x$ and $y$ occur in $\word$.

    Let $\hlId$ be a balanced non-trivial identity satisfying \hlDo{}, \hlDt{}, \hlB{} and \hlC{}, of length $n \geq 5$, with common suffix $\word \in \calX^*$ such that $|\word| \leq n-3$. Notice that no variable in $x \vord_2'$ is simple by \hlC{}, thus $x$ must occur in $\vord_1'$ or $\word$. 
    
    If $x$ does not occur in $\vord_1'$, then we have $\supp{\vord_2'} \subseteq \supp{\vord_1'} \cap \supp{\word}$ by \hlDo{} and \hlB{}. Thus, the identity $\vord_1' x \vord_2' \word \approx \vord_1'\vord_2' x \word$ is a consequence of a subset of \hrOot{} and \hrEot{}, depending on where $x$ and the variables of $\vord_2'$ occur in $\word$.

    On the other hand, if $x$ does not occur in $\word$, then $\supp{\vord_2'} \subseteq \supp{\vord_1'} \cap \supp{\word}$ by \hlB{} and \hlDt{}. Thus, $\vord_1' x \vord_2' \word \approx \vord_1'\vord_2' x \word$ is a consequence of a subset of \hrOto{} and \hrEto{}, depending on where $x$ and the variables of $\vord_2'$ occur in $\vord_1'$.

    Finally, if $x$ occurs in both $\vord_1'$ and $\word$, then $\vord_1' x \vord_2' \word \approx \vord_1'\vord_2' x \word$ is a consequence of a subset of \hrOot{}, \hrEot{}, \hrOto{} and \hrEto{}, depending on where $x$ and the variables of $\vord_2'$ occur in $\vord_1'$ or $\word$. The result follows.
\end{proof}

\begin{corollary} \label{cor:vlst_join_VarS_vrst_join_VarS_identities} 
    The equational theory of $\vlst \vee \VarS$ is the set of balanced identities that satisfy the properties \hlB{}, \hlC{} and \hlDo{}, and the equational theory of $\vrst \vee \VarS$ is the set of balanced identities that satisfy the properties \hlB{}, \hlC{} and \hlDt{}.
\end{corollary}
\begin{proof}
    Follows from Corollary~\ref{cor:vlst_vrst_identities} and Proposition~\ref{prop:VarS_identities}.
\end{proof}

\begin{proposition}\label{prop:vlst_join_VarS_vrst_join_VarS_finite_basis}
    The variety $\vlst \vee \VarS$ admits a finite equational basis consisting of the identities \hrOot{}, \hrEot{}, and \hrLt{}, and the variety $\vrst \vee \VarS$ admits a finite equational basis consisting of the identities \hrOto{}, \hrEto{}, and \hrRt{}.
\end{proposition}
\begin{proof}
    Clearly, the identities \hrOot{}, \hrEot{} and \hrLt{} satisfy properties \hlDo{}, \hlB{} and \hlC{}, which define the equational theory of $\vlst \vee \VarS$ by Corollary~\ref{cor:vlst_join_VarS_vrst_join_VarS_identities}. Now, we prove by induction that any identity satisfying these properties is a consequence of said identities.

    The base case for the induction is the identities of the form
    \[ xyx \word \approx xxy \word. \]
    Notice that $x$ and $y$ must occur in $\word$ by \hlB{} and \hlC{}, respectively. Thus, the identity is a consequence of either \hrOot{} and \hrEot{}, depending on where $x$ and $y$ occur in $\word$.

    Let $\hlId$ be a balanced non-trivial identity satisfying \hlDo{}, \hlB{} and \hlC{}, of length $n \geq 4$, with common suffix $\word \in \calX^*$ such that $|\word| \leq n-3$. Notice that no variable in $x \vord_2'$ is simple by \hlC{}, thus $x$ must occur in $\vord_1'$ or $\word$. 
    If $x$ does not occur in $\vord_1'$, then we have $\supp{\vord_2'} \subseteq \supp{\vord_1'} \cap \supp{\word}$ by \hlDo{} and \hlB{}. Thus, the identity $\vord_1' x \vord_2' \word \approx \vord_1'\vord_2' x \word$ is a consequence of a subset of \hrOot{} and \hrEot{}, depending on where $x$ and the variables of $\vord_2'$ occur in $\word$.

    On the other hand, if $x$ does not occur in $\word$, then $\supp{\vord_2'} \subseteq \supp{\vord_1'}$ by \hlB{}. Thus, $\vord_1' x \vord_2' \word \approx \vord_1'\vord_2' x \word$ is a consequence of \hrLt{}.

    Finally, if $x$ occurs in both $\vord_1'$ and $\word$, then $\vord_1' x \vord_2' \word \approx \vord_1'\vord_2' x \word$ is a consequence of a subset of \hrOot{}, \hrEot{} and \hrLt{}, depending on where $x$ and the variables of $\vord_2'$ occur in $\vord_1'$ or $\word$. The result follows. A dual argument works for the case of $\vrst \wedge \VarS$.
\end{proof}

\subsection{Proving Theorem~\ref{theorem:lattice_1}} \label{subsection:L1_correctness}

The following results allow us to conclude the proof of Theorem~\ref{theorem:lattice_1}, by showing that the lattice \hrLato{} has no missing varieties, and the covers are well-determined. We only state the results that do not follow immediately from algebraic manipulation or previously established strict containments of varieties. Recall that, if any three elements $a,b,c$ of a partial order are such that $a<b$ and $c$ is incomparable with both $a$ and $b$, then $c$ is also incomparable with any element between $a$ and $b$, and if $b < a \vee c$ (resp. $a > b \wedge c$), then $b \vee c = a \vee c$ (resp. $a \wedge c = b \wedge c$).

\begin{lemma} \label{lemma:incomparabilities_lattice_1}
    The following statements hold:
    \begin{enumerate}[label=(\roman*)]
        \item $\vmst \wedge \VarS$ is incomparable with $\vlst$ and $\vrst$;
        \item $\vmst \vee \VarS$ is incomparable with $\vsylvh$ and $\vsylv$;
        \item $\vlst \wedge \vsylv$ is incomparable with $\vrst \wedge \vsylvh$ and $\vrst$;
        \item $\vrst \wedge \vsylvh$ is incomparable with $\vlst \wedge \vsylv$ and $\vlst$;
        \item $\vrst \vee \vsylvh$ is incomparable with $\vlst \vee \vsylv$ and $\vsylv$;
        \item $\vlst \vee \vsylv$ is incomparable with $\vrst \vee \vsylvh$ and $\vsylvh$;
        \item $\vmst \wedge \vsylvh$ is incomparable with $\VarS$, $\vrst$ and $\vsylv$;
        \item $\vmst \wedge \vsylv$ is incomparable with $\VarS$, $\vlst$ and $\vsylvh$;
        \item $\vlst \vee \VarS$ is incomparable with $\vmst$, $\vrst$ and $\vsylv$;
        \item $\vrst \vee \VarS$ is incomparable with $\vmst$, $\vlst$ and $\vsylvh$.
    \end{enumerate}
\end{lemma}
\begin{proof}
    By Proposition~\ref{prop:vmst_meet_VarS_identities}, the identity $xy^2x \approx yx^2y$ is satisfied by $\vmst \wedge \VarS$, while neither \hrLo{} nor \hrRo{} are. Moreover, by Corollary~\ref{cor:vlst_vrst_identities}, \hrLo{} is satisfied by $\vlst$, \hrRo{} is satisfied by $\vrst$, and $xy^2x \approx yx^2y$ is satisfied by neither of them. Therefore, \textit{(i)} holds.

    By Corollary~\ref{cor:vmst_join_VarS_identities}, the identity $xyxyxy \approx x^2yxy^2$ is satisfied by $\vmst \vee \VarS$, but neither \hrLt{} nor \hrRt{} are. However, by Theorem~\ref{theorem:vsylvh_vsylv_identities}, \hrRt{} is satisfied by $\vsylv$, \hrLt{} is satisfied by $\vsylvh$, and $xyxyxy \approx x^2yxy^2$ is satisfied by neither of them. Thus, \textit{(ii)} holds.

    By Corollary~\ref{cor:vjst_identities} and Proposition~\ref{prop:vlst_meet_vsylv_vrst_meet_vsylvh_identities}, \hrLo{} is satisfied by $\vlst \wedge \vsylv$ but neither by $\vrst \wedge \vsylvh$ nor $\vrst$, and \hrRo{} is satisfied by $\vrst \wedge \vsylvh$ and $\vrst$ but not by $\vlst \wedge \vsylv$. Hence \textit{(iii)} holds. Case \textit{(iv)} holds by a dual argument.

    By Theorem~\ref{theorem:vbaxt_identities} and Corollary~\ref{cor:vrst_join_vsylvh_vlst_join_vsylv_identities}, \hrOto{} is satisfied by $\vrst \wedge \vsylvh$, but neither by $\vlst \vee \vsylv$ nor $\vsylv$, and \hrOot{} is satisfied by $\vlst \wedge \vsylv$ and $\vsylv$, but not by $\vrst \vee \vsylvh$. As such, \textit{(v)} holds. Case \textit{(vi)} holds by a dual argument.

    By Proposition~\ref{prop:vmst_meet_vsylvh_vmst_meet_vsylv_identities}, \hrLt{} and \hrMd{} are satisfied by $\vmst \wedge \vsylvh$. By Proposition~\ref{prop:VarS_identities}, \hrMd{} is not satisfied by $\VarS$, and by Theorem~\ref{theorem:vsylvh_vsylv_identities} and Corollary~\ref{cor:vlst_vrst_identities}, \hrLt{} is satisfied by neither $\vrst$ nor $\vsylv$. On the other hand, by the same results, \hrRt{} is satisfied by $\VarS$, $\vrst$ and $\vsylv$, but not by $\vmst \wedge \vsylvh$. Therefore, \textit{(vii)} holds. Case \textit{(viii)} holds by a dual argument.

    By Corollary~\ref{cor:vlst_join_VarS_vrst_join_VarS_identities}, \hrLt{} is satisfied by $\vlst \vee \VarS$. By Corollary~\ref{cor:vmst_identities} and the results mentioned in the previous paragraph, \hrLt{} is satisfied by neither $\vmst$, $\vrst$ nor $\vsylv$. On the other hand, \hrRt{} is satisfied by $\vrst$ and $\vsylv$, but not by $\vmst \wedge \vsylvh$. Furthermore, \hrMd{} is satisfied by $\vmst$, but not by $\vlst \vee \VarS$. Therefore, \textit{(ix)} holds. Case \textit{(x)} holds by a dual argument.
\end{proof}

\begin{lemma} \label{lemma:equalities_lattice_1}
    The following equalities hold:
    \begin{enumerate}[label=(\roman*)]
        \item $\vlst \wedge \vsylv = (\vmst \wedge \VarS) \wedge \vlst$;
        \item $\vrst \wedge \vsylvh = (\vmst \wedge \VarS) \wedge \vrst$;
        \item $\vmst \wedge \VarS = (\vlst \wedge \vsylv) \vee (\vrst \wedge \vsylvh)$;
        \item $\vmst \wedge \vsylvh = \vlst \vee (\vrst \wedge \vsylvh) = \vmst \wedge (\vlst \vee \VarS)$;
        \item $\vmst \wedge \vsylv = \vrst \vee (\vlst \wedge \vsylv) = \vmst \wedge (\vrst \vee \VarS)$;
        \item $\vrst \vee \vsylvh = (\vmst \vee \VarS) \vee \vsylvh$;
        \item $\vlst \vee \vsylv = (\vmst \vee \VarS) \vee \vsylv$;
        \item $\vmst \vee \VarS = (\vrst \vee \vsylvh) \wedge (\vlst \vee \vsylv)$;
        \item $\vlst \vee \VarS = \vsylvh \wedge (\vlst \vee \vsylv) = (\vmst \wedge \vsylvh) \vee \VarS$;
        \item $\vrst \vee \VarS = \vsylv \wedge (\vrst \vee \vsylvh) = (\vmst \wedge \vsylv) \vee \VarS$.
    \end{enumerate}
\end{lemma}
\begin{proof}
    Cases \textit{(i)} and \textit{(ii)} follow from Theorem~\ref{theorem:vsylvh_vsylv_finite_basis} and Corollaries~\ref{cor:vlst_vrst_finite_basis} and \ref{cor:vmst_meet_VarS_finite_basis}, as \hrLt{}, \hrMd{}, and \hrMt{} are consequences of \hrLo{}, and \hrRt{}, \hrMd{}, and \hrMt{} are consequences of \hrRo{}.

    Case \textit{(iii)} follows from Propositions~\ref{prop:vlst_meet_vsylv_vrst_meet_vsylvh_identities} and \ref{prop:vmst_meet_VarS_identities}. 
    
    The first equalities in cases \textit{(iv)} and \textit{(v)} follow from Corollary~\ref{cor:vlst_vrst_identities} and Propositions~\ref{prop:vmst_meet_VarS_identities} and \ref{prop:vmst_meet_vsylvh_vmst_meet_vsylv_identities} as \hlDo{} implies \hlEo{} and \hlDt{} implies \hlEt{}. The second equalities follow from Corollary~\ref{cor:vmst_finite_basis} and Proposition~\ref{prop:vlst_join_VarS_vrst_join_VarS_finite_basis} as \hrOto{}, \hrEto{},\hrOot{} and \hrEot{} are consequences of either \hrMd{} or \hrMt{}.

    Cases \textit{(vi)} and \textit{(vii)} follow from Theorem~\ref{theorem:vsylvh_vsylv_identities} and Corollaries~\ref{cor:vrst_join_vsylvh_vlst_join_vsylv_identities} and \ref{cor:vmst_join_VarS_identities} as \hlAo{} implies \hlDo{}, \hlB{}, and \hlC{}, and \hlAt{} implies \hlDt{}, \hlB{}, and \hlC{}.
    
    Case \textit{(viii)} follows from Propositions~\ref{prop:vrst_join_vsylvh_vlst_join_vsylv_finite_basis} and \ref{prop:vmst_join_VarS_finite_basis}. 
    
    The first equalities in cases \textit{(ix)} and \textit{(x)} follow from Theorem~\ref{theorem:vsylvh_vsylv_finite_basis} and Propositions~\ref{prop:vmst_join_VarS_finite_basis} and \ref{prop:vlst_join_VarS_vrst_join_VarS_finite_basis}, as \hrOto{} and \hrEto{} are consequences of \hrLt{}, and \hrOot{} and \hrEot{} are consequences of \hrRt{}.  The second equalities follow from Propositions~\ref{prop:VarS_identities} and \ref{prop:vmst_meet_vsylvh_vmst_meet_vsylv_identities} as \hlB{} and \hlC{} both imply \hlEo{} and \hlEt{}.
\end{proof}

The correctness of Theorem~\ref{theorem:lattice_1} then follows from Lemmas~\ref{lemma:incomparabilities_lattice_1} and \ref{lemma:equalities_lattice_1}.

\subsection{Axiomatic ranks} \label{subsection:L1_axiomatic_ranks}

\begin{lemma}
\label{lemma:vmst_meet_VarS_shortest_identity}
    The shortest non-trivial identity, with n variables, satisfied by $\vmst \wedge \VarS$, is of length $n+2$.
\end{lemma}
\begin{proof} 
    Let $\uord \approx \vord$ be a non-trivial identity, with $|\supp{\uord \approx \vord}|=n$, satisfied by $\vmst \wedge \VarS$. By Proposition~\ref{prop:vmst_meet_VarS_identities}, $\uord \approx \vord$ is balanced, thus we can write $\uord = \word x \uord'$ and $\vord=\word y \vord'$, where $x,y \in \calX$ and $\word,\uord',\vord' \in \calX^*$ are such that $x$ occurs in $\vord'$ and $y$ occurs in $\uord'$. Furthermore, if $x$ is simple, then $y$ also occurs in $\word$ by \hlEo{}, and in $\vord'$ (after $x$) by \hlEt{}. As such, at least one variable is non-simple, and if it is the only non-simple variable, it must occur at least three times, from which we conclude that the length of $\uord \approx \vord$ is at least $n+2$.
    
    On the other hand, by Proposition~\ref{prop:vmst_meet_VarS_identities}, for variables $x,y,a_1, \dots, a_{n-2} \in \calX$, the identity
    \[ x y x a_1 \cdots a_{n-2} x \approx x x y a_1 \cdots a_{n-2} x, \]
    of length $n+2$, is satisfied by $\vmst \wedge \VarS$. The result follows.
\end{proof}

The proof of the following result uses the same technique as the one given in the proof of \cite[Proposition~7.9]{aird_ribeiro_join_meet_stalactic_2023}, using Proposition~\ref{prop:vmst_meet_VarS_identities} and Lemma~\ref{lemma:vmst_meet_VarS_shortest_identity} accordingly, so we omit it.

\begin{lemma} \label{lemma:L2_Md_R2_not_consequences} 
    None of the identities \hrLt{}, \hrMd{} or \hrRt{} is a consequence of the set of non-trivial identities, satisfied by $\vmst \wedge \VarS$, over an alphabet with four variables, excluding itself and equivalent identities. 
\end{lemma}

\begin{lemma}
\label{lemma:vlst_join_VarS_vrst_join_VarS_shortest_identity}
    The shortest non-trivial identity, with n variables, satisfied by $\vlst \vee \VarS$ or $\vrst \vee \VarS$, is of length $n+2$.
\end{lemma}
\begin{proof}
    This follows from Lemma~\ref{lemma:vmst_meet_VarS_shortest_identity}, which gives the lower bound, and Corollary~\ref{cor:vlst_join_VarS_vrst_join_VarS_identities}, since for variables $x,y,a_1, \dots, a_{n-2} \in \calX$, the identity
    \[ x y a_1 \cdots a_{n-2} x y \approx x y a_1 \cdots a_{n-2} y x, \]
    of length $n+2$, is satisfied by $\vlst \vee \VarS$.
\end{proof}

\begin{lemma} \label{lemma:O12_E12_O21_E21_not_consequences}
    Neither of the identities \hrOot{}, \hrEot{}, \hrOto{} or \hrEto{} is a consequence of the set of non-trivial identities satisfied by $\vlst \vee \VarS$, in the case of the first two, and $\vrst \vee \VarS$ otherwise, over an alphabet with five variables, excluding itself and equivalent identities.
\end{lemma}
\begin{proof}
    Let $\mathcal{S}$ be the set of all non-trivial identities, satisfied by $\vlst \vee \VarS$, over an alphabet with five variables. By definition, \hrOot{} is a consequence of $\mathcal{S}$. As such, there exists a non-trivial identity $\uord \approx \vord$ in $\mathcal{S}$, and a substitution $\psi$, such that
    \[
    xzxytxry = \word_1 \psi(\uord) \word_2,
    \]
    where $\word_1, \word_2$ are words over the five-variable alphabet, and $\psi(\uord) \neq \psi(\vord)$. We can assume, without loss of generality, that $\psi$ does not map any variable to the empty word.
    By Lemma~\ref{lemma:vlst_join_VarS_vrst_join_VarS_shortest_identity}, any proper factor of $xzxytxry$ is an isoterm for $\vmst \wedge \VarS$, with the possible exceptions of $xzxytxr$, $zxytxry$, $xzxytx$ and $xytxry$. By Corollary~\ref{cor:vlst_join_VarS_vrst_join_VarS_identities}, in particular property \hlC{}, it is clear that if $xzxytxr$ (resp. $zxytxry$) is not an isoterm, then $xzxytx$ (resp. $xytxry$) is also not. On one hand, $xzxytx$ is an isoterm by \hlC{}, since the only non-simple variable is $x$. On the other, $xytxry$ is an isoterm by \hlC{} and \hlDo{}. Hence, $\word_1$ and $\word_2$ are the empty word, that is, $xzxytxry = \psi(\uord)$.
    
    Since $\cont{xzxytxry}=\bigl(\begin{smallmatrix} x & y & z & t & r \\ 3 & 2 & 1 & 1 & 1 \end{smallmatrix}\bigr)$ and, by Lemma~\ref{lemma:vlst_join_VarS_vrst_join_VarS_shortest_identity}, there is a lower bound on the length of identities satisfied by $\vlst \vee \VarS$, we can conclude that, up to renaming of variables, $x$ occurs at least twice and at most thrice, $y$ occurs at least once and at most twice, and $z$, $t$ and $r$ can each occur at most once in $\uord \approx \vord$. Furthermore, if $x$ occurs only twice, then $y$ occurs twice and, on the other hand, if $y$ occurs only once, then $x$ occurs thrice.
    
    Notice that, since all factors of length $2$ of $xzxytxry$ are distinct, then $\psi(x) = x$ and $\psi(y) = y$. As such, $\psi(ztr)$ can have at most one occurrence of $x$ or $y$. Thus, since the shortest factor of $xzxytxry$ where $z$ and $t$ occur is $zxyt$, at least $z$ and $t$ (or $r$) occur in $\uord \approx \vord$. Assume, without loss of generality, that $t$ occurs.
    
    Suppose $x$ only occurs twice in $\uord \approx \vord$. If $x$ occurs in $\psi(z)$, then $\psi(t) = t$, which implies that $r$ occurs in $\uord \approx \vord$ and $\psi(r) = r$. But $\psi(z) \neq xz$ since $zxytxry$ is an isoterm and $\psi(z) = zx$ since $xzytxry$ is an isoterm by \hlDo{}, hence $x$ cannot occur in $\psi(z)$. On the other hand, if $x$ occurs in $\psi(t)$, then $\psi(t) = tx$ or $\psi(t) = txr$, which is impossible since $xzxyty$ is an isoterm by \hlB{}. By the same reasoning, if $r$ occurs in $\uord \approx \vord$, $x$ cannot occur in $\psi(r)$. Therefore, $x$ must occur thrice in $\uord \approx \vord$. Then, since all factors of $xzxytxry$ where at least two of $z$, $t$ or $r$ occur also have an occurrence of $x$, we can conclude that $z$, $t$ and $r$ all occur in $\uord \approx \vord$. Furthermore, $\psi(z) = z$.

    Suppose now that $y$ occurs only once in $\uord \approx \vord$. Then, either $\psi(t) = yt$ or $\psi(r) = ry$. Neither case can happen, since $xzxtxry$ and $xzxytx$ are isoterms by \hlC{}. Thus, $y$ occurs twice in $\uord \approx \vord$, and $\psi(t) = t$ and $\psi(r) = r$.
    
    As such, since we are considering only substitutions that do not map variables to the empty word, we have that $\psi$ is only a renaming of variables. Notice that $xzxytxry$ is the left-hand side of a non-trivial identity satisfied by $\vlst \vee \VarS$ if and only if the right-hand side is $xzyxtxry$, by Corollary~\ref{cor:vlst_join_VarS_vrst_join_VarS_identities}. Hence $\uord \approx \vord$ is equivalent to \hrOot{}.
    
    The proof for the case of \hrEot{} follows a similar reasoning. By dual reasoning, we prove the cases of \hrOto{} and \hrEto{}.
\end{proof}

\begin{corollary} \label{cor:L1_lattice_axiomatic_rank}
    The axiomatic rank of the varieties
    \begin{enumerate}[label=(\roman*)]
        \item $\vlst \wedge \vsylv$ and $\vrst \wedge \vsylvh$ is 2; 
        \item $\VarS$, $\vmst \wedge \vsylvh$, $\vmst \wedge \vsylv$, and $\vmst \wedge \VarS$ is 4;
        \item  $\vrst \vee \vsylvh$, $\vlst \vee \vsylv$, $\vmst \vee \VarS$, $\vlst \vee \VarS$, and $\vrst \vee \VarS$ is 5.
    \end{enumerate}
\end{corollary}
\begin{proof}
    Case \textit{(i)} follows from Proposition~\ref{prop:vlst_meet_vsylv_vrst_meet_vsylvh_finite_basis} and these varieties being overcommutative, while case \textit{(ii)} follows from Corollaries~\ref{cor:VarS_finite_basis}, \ref{cor:vmst_meet_VarS_finite_basis} and \ref{cor:vmst_meet_vsylvh_vmst_meet_vsylv_finite_basis} and Lemma~\ref{lemma:L2_Md_R2_not_consequences}, and case \textit{(iii)} follows from Propositions~\ref{prop:vrst_join_vsylvh_vlst_join_vsylv_finite_basis}, \ref{prop:vmst_join_VarS_finite_basis} and \ref{prop:vlst_join_VarS_vrst_join_VarS_finite_basis} and Lemma~\ref{lemma:O12_E12_O21_E21_not_consequences}.
\end{proof}

\section{Sublattice of $\mathbb{MON}$ generated by $\vsylvh$, $\vsylv$, $\vlst$, $\vrst$ and $\vhypo$.} \label{section:lattice_L2}

We now consider the lattice \hrLatd{}, obtained from \hrLato{} by adding a new generator $\vhypo$.

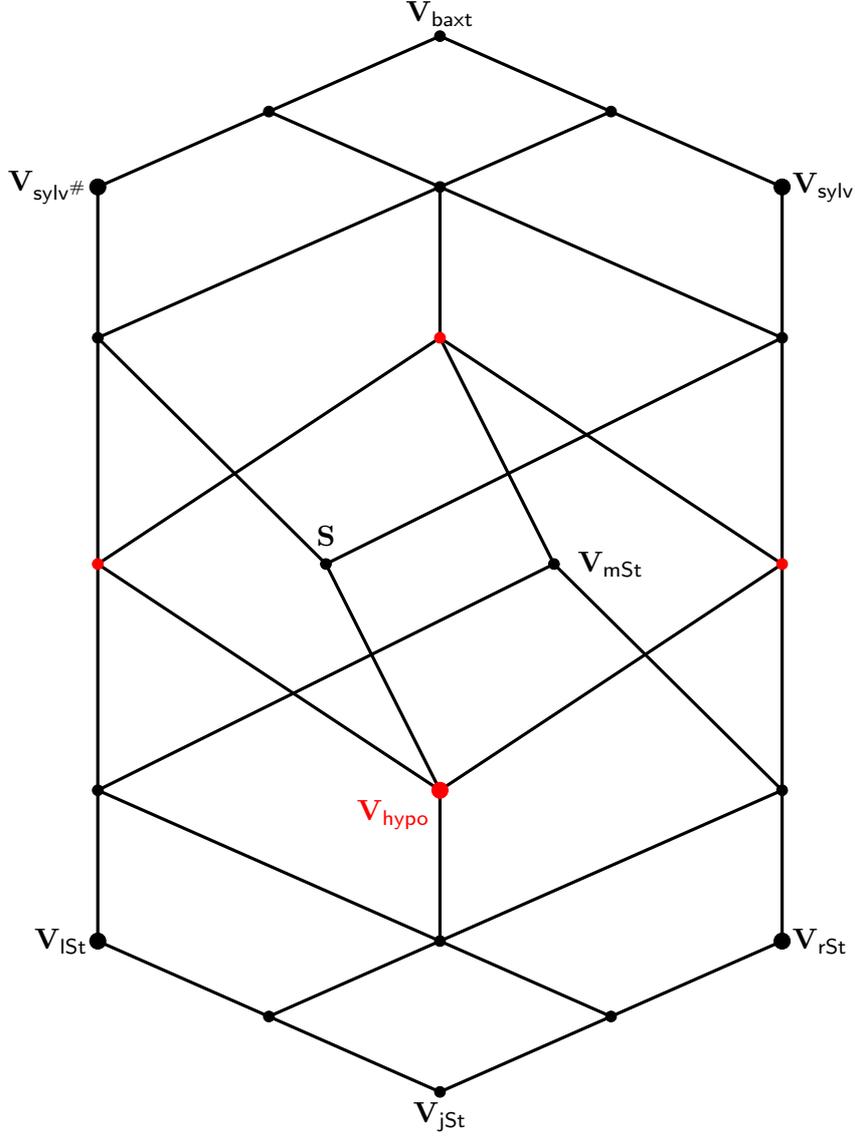
\begin{figure}[h]
    \centering
    \begin{tikzpicture}
        \draw[very thick] (0,1) -- (-2.25,0) -- (-4.5,-1) -- (-4.5,-3) -- (-4.5,-6) -- (-4.5,-11) -- (-2.25,-12) -- (0,-13);
        \draw[very thick] (0,1) -- (2.25,0) -- (4.5,-1) -- (4.5,-3) -- (4.5,-6) -- (4.5,-11) -- (2.25,-12) -- (0,-13);
        \draw[very thick] (-2.25,0) -- (0,-1) -- (-4.5,-3) -- (-1.5,-6) -- (0,-9) -- (0,-11) -- (-2.25,-12);
        \draw[very thick] (2.25,0) -- (0,-1) -- (4.5,-3) -- (-1.5,-6);
        \draw[very thick] (0,-1) -- (0,-3) -- (-4.5,-6) -- (0,-9);
        \draw[very thick] (0,-3) -- (1.5,-6) -- (4.5,-9) -- (0,-11) -- (2.25,-12);
        \draw[very thick] (0,-3) -- (4.5,-6) -- (0,-9);
        \draw[very thick] (1.5,-6) -- (-4.5,-9) -- (0,-11);
        \filldraw[black] (0,1) circle (2pt) node[anchor=south]{$\vbaxt$};
        \filldraw[black] (-2.25,0) circle (2pt);
        \filldraw[black] (2.25,0) circle (2pt);
        \filldraw[black] (-4.5,-1) circle (3pt) node[anchor=east]{$\vsylvh$};
        \filldraw[black] (0,-1) circle (2pt);
        \filldraw[black] (4.5,-1) circle (3pt) node[anchor=west]{$\vsylv$};
        \filldraw[black] (-4.5,-3) circle (2pt);
        \filldraw[red] (0,-3) circle (2pt);
        \filldraw[black] (4.5,-3) circle (2pt);
        \filldraw[red] (-4.5,-6) circle (2pt);
        \filldraw[black] (-1.5,-6) circle (2pt) node[above=3pt]{$\VarS$};
        \filldraw[black] (1.5,-6) circle (2pt) node[right=5pt]{$\vmst$};
        \filldraw[red] (4.5,-6) circle (2pt);
        \filldraw[black] (-4.5,-9) circle (2pt);
        \filldraw[red] (0,-9) circle (3pt) node[anchor=north east]{$\vhypo$};
        \filldraw[black] (4.5,-9) circle (2pt);
        \filldraw[black] (-4.5,-11) circle (3pt) node[anchor=east]{$\vlst$};
        \filldraw[black] (4.5,-11) circle (3pt) node[anchor=west]{$\vrst$};
        \filldraw[black] (0,-11) circle (2pt);
        \filldraw[black] (-2.25,-12) circle (2pt);
        \filldraw[black] (2.25,-12) circle (2pt);
        \filldraw[black] (0,-13) circle (2pt) node[anchor=north]{$\vjst$};
    \end{tikzpicture}
    \caption{Sublattice $\mathbb{L}_2$ of $\mathbb{MON}$ generated by the varieties respectively generated by the \#-sylvester, sylvester, left and right stalactic, and hypoplactic monoids. The points in red correspond to the elements that are not in \hrLato{}.}
    \label{fig:lattice_with_sylv_stal_hypo}
\end{figure}

\begin{theorem} \label{theorem:lattice_2}
    The Hasse diagram of \hrLatd{} is given in Figure~\ref{fig:lattice_with_sylv_stal_hypo}.
\end{theorem}

\subsection{Varietal meets and joins} \label{subsection:L2_meets_joins}

\begin{corollary} \label{cor:vhypo_join_vlst_vhypo_join_vrst_identities}
    The equational theory of $\vhypo \vee \vlst$ is the set of balanced identities that satisfy the properties \hlB{} and \hlDo{}, and the equational theory of $\vhypo \vee \vrst$ is the set of balanced identities that satisfy the properties \hlB{} and \hlDt{}.
\end{corollary}
\begin{proof}
    Follows from Theorem~\ref{theorem:vhypo_identities} and Corollary~\ref{cor:vlst_vrst_identities}.
\end{proof}

\begin{corollary} \label{cor:vhypo_join_vlst_vhypo_join_vrst_varietal_join}
    The variety $\vhypo \vee \vlst$ is the varietal join of $\mathbf{COM}$ and $\mathbf{J}_2 \vee \mathbf{LRB}$, and the variety $\vhypo \vee \vrst$ is the varietal join of $\mathbf{COM}$ and $\mathbf{J}_2 \vee \mathbf{RRB}$.
\end{corollary}
\begin{proof}
    Follows from Corollaries~\ref{cor:vhypo_varietal_join} and \ref{cor:vlst_vrst_varietal_join}.
\end{proof}

\begin{corollary} \label{cor:vhypo_join_vlst_vhypo_join_vrst_generators}
    The variety $\vhypo \vee \vlst$ is generated by the factor monoid $\bbN^*/{(\equiv_\hypo \wedge \equiv_\lst)}$, and the variety $\vhypo \vee \vrst$ is generated by the factor monoid $\bbN^*/{(\equiv_\hypo \wedge \equiv_\rst)}$.
\end{corollary}

\begin{proposition} \label{prop:vhypo_join_vlst_vhypo_join_vrst_finite_basis}
    The variety $\vhypo \vee \vlst$ admits a finite equational basis consisting of the identities \hrMt{} and \hrLt{}, and the variety $\vhypo \vee \vrst$ admits a finite equational basis consisting of the identities \hrMt{} and \hrRt{}.
\end{proposition}
\begin{proof}
    Clearly, the identities \hrMt{} and \hrLt{} satisfy properties \hlDo{} and \hlB{}, which define the equational theory of $\vhypo \vee \vlst$ by Corollary~\ref{cor:vhypo_join_vlst_vhypo_join_vrst_identities}. Now, we prove by induction that any identity satisfying these properties is a consequence of said identities.

    The base case for the induction is the identities of the form
    \[ xxy \word \approx xyx \word. \]
    Since $x$ must occur in $\word$ by \hlB{}, the identity is a consequence of \hrMt{}.

    Let $\hlId$ be a balanced non-trivial identity satisfying \hlDo{} and \hlB{}, of length $n \geq 4$, with common suffix $\word \in \calX^*$ such that $|\word| \leq n-3$. Clearly, if $x$ occurs in both $\vord_1'$ and $\word$, then the identity $\vord_1' x \vord_2' \word \approx \vord_1'\vord_2' x \word$ is a consequence of \hrMt{}. In this case, there are only two deduction steps: we have 
    \[ \vord_1' x \vord_2' \word \approx \vord_1'' x^2 \vord_2'' \word \approx \vord_1'\vord_2' x \word, \]
    for $\vord_1'' \in \calX^*$, $\vord_2'' \in (\calX\setminus \{x\})^+$. 
    
    If $x$ is simple or it occurs in $\word$ but not in $\vord_1'$, then $\supp{\vord_2'} \subseteq \supp{\vord_1'} \cap \supp{\word}$ by \hlB{} alone in the first case and by and \hlDo{} as well in the second. As such, $\vord_1' x \vord_2' \word \approx \vord_1'\vord_2' x \word$ is a consequence of \hrMt{}. 

    On the other hand, if $x$ occurs in $\vord_1'$ but not in $\word$, then we have $\supp{\vord_2'} \subseteq \supp{\vord_1'}$ by \hlB{}. Thus, $\vord_1' x \vord_2' \word \approx \vord_1'\vord_2' x \word$ is a consequence of \hrLt{}. The result follows. A dual argument works for the case of $\vhypo \vee \vrst$.
\end{proof}

\begin{corollary} \label{cor:vhypo_join_vmst_identities}
    The equational theory of $\vhypo \vee \vmst$ is the set of balanced identities that satisfy the properties \hlB{}, \hlDo{} and \hlDt{}.
\end{corollary}
\begin{proof}
    Follows from Theorem~\ref{theorem:vhypo_identities} and Corollary~\ref{cor:vmst_identities}.
\end{proof}

\begin{corollary} \label{cor:vhypo_join_vmst_varietal_join}
    The variety $\vhypo \vee \vmst$ is the varietal join of $\mathbf{COM}$ and $\mathbf{J}_2 \vee \mathbf{RB}$.
\end{corollary}
\begin{proof}
    Follows from Corollaries~\ref{cor:vhypo_varietal_join} and \ref{cor:vmst_varietal_join}.
\end{proof}

\begin{corollary} \label{cor:vhypo_join_vmst_generators}
    The variety $\vhypo \vee \vmst$ is generated by the factor monoid $\bbN^*/{(\equiv_\hypo \wedge \equiv_\mst)}$.
\end{corollary}

\begin{proposition} \label{prop:vhypo_join_vmst_finite_basis}
    The variety $\vhypo \vee \vmst$ admits a finite equational basis consisting of the identity \hrMt{}.
\end{proposition}
\begin{proof}
    Clearly, the identity \hrMt{} satisfies properties \hlDo{}, \hlDt{} and \hlB{}, which define the equational theory of $\vhypo \vee \vmst$ by Corollary~\ref{cor:vhypo_join_vmst_identities}. Now, we prove by induction that any identity satisfying these properties is a consequence of said identity.

    The base case for the induction is the identities of the form
    \[ xxy \word \approx xyx \word. \]
    Since $x$ must occur in $\word$ by \hlB{}, the identity is a consequence of \hrMt{}.

    Let $\hlId$ be a balanced non-trivial identity satisfying \hlDo{}, \hlDt{} and \hlB{}, of length $n \geq 4$, with common suffix $\word \in \calX^*$ such that $|\word| \leq n-3$. If $x$ does not occur in both $\vord_1'$ and $\word$, then all variables in $\vord_2'$ do: if $x$ is simple, by \hlB{}; if $x$ occurs in $\vord_1'$ but not in $\word$, by \hlB{} and \hlDt{}; if $x$ occurs in $\word$ but not in $\vord_1'$, by \hlDo{} and \hlB{}.
        
    Thus, the identity $\vord_1' x \vord_2' \word \approx \vord_1'\vord_2' x \word$ is a consequence of \hrMt{}. 
    The result follows.
\end{proof}

\subsection{Proving Theorem~\ref{theorem:lattice_2}} \label{subsection:L2_correctness}
In order to prove Theorem~\ref{theorem:lattice_2}, we require the following two lemmas that show that the lattice \hrLatd{} has no missing varieties, and all the covers are well-determined.

\begin{lemma} \label{lemma:incomparabilities_lattice_2}
    The following statements hold
    \begin{enumerate}[label=(\roman*)]
        \item $\vhypo$ is incomparable with $\vlst$, $\vrst$, and $\vmst$;
        \item $\vhypo \vee \vlst$ is incomparable with $\VarS$, $\vmst$, $\vrst$ and $\vsylv$;
        \item $\vhypo \vee \vrst$ is incomparable with $\VarS$, $\vmst$, $\vlst$ and $\vsylvh$;
        \item $\vhypo \vee \vmst$ is incomparable with $\VarS$, $\vsylvh$ and $\vsylv$.
    \end{enumerate}
\end{lemma}
\begin{proof}
    By Theorem~\ref{theorem:vhypo_identities} and Corollary~\ref{cor:vhypo_join_vlst_vhypo_join_vrst_identities}, the identity \hrLt{} is satisfied by $\vhypo$ and $\vhypo \vee \vlst$ and \hrRt{} is satisfied by $\vhypo$ and $\vhypo \vee \vrst$, but, by Theorem~\ref{theorem:vsylvh_vsylv_identities} and Corollaries~\ref{cor:vlst_vrst_identities} and \ref{cor:vmst_identities}, \hrLt{} is satisfied by neither $\vsylv$, $\vrst$ nor $\vmst$, and \hrRt{} is satisfied by neither $\vsylvh$, $\vlst$, nor $\vmst$. Moreover, by Corollary~\ref{cor:vhypo_join_vmst_identities}, the identity \hrMt{} is satisfied by $\vhypo \vee \vlst$, $\vhypo \vee \vrst$ and $\vhypo \vee \vmst$, but neither by $\vsylvh$, $\vsylv$ nor $\VarS$, due to Proposition~\ref{prop:VarS_identities}.

    \par Similarly, by Theorem~\ref{theorem:vsylvh_vsylv_identities} and Proposition~\ref{prop:VarS_identities}, the identity \hrLt{} is satisfied by $\vsylvh$ and $\VarS$, and \hrRt{} is satisfied by $\vsylv$ and $\VarS$, but, by Corollaries~\ref{cor:vhypo_join_vlst_vhypo_join_vrst_identities} and \ref{cor:vhypo_join_vmst_identities}, \hrLt{} is satisfied by neither $\vhypo \vee \vrst$ nor $\vhypo \vee \vmst$ and \hrRt{} is satisfied by neither $\vhypo \vee \vlst$ nor $\vhypo \vee \vmst$. Moreover, by Corollaries~\ref{cor:vlst_vrst_identities} and \ref{cor:vmst_identities}, the identity \hrMd{} is satisfied by $\vlst$, $\vrst$ and $\vmst$, but, by Theorem~\ref{theorem:vhypo_identities}, is satisfied by neither $\vhypo$, $\vhypo \vee \vlst$ nor $\vhypo \vee \vrst$.
\end{proof}

\begin{lemma} \label{lemma:equalities_lattice_2}
    The following equalities hold:
    \begin{enumerate}[label=(\roman*)]
        \item $\vhypo = \VarS \wedge (\vhypo \vee \vmst) = (\vhypo \vee \vlst) \wedge \vsylv = (\vhypo \vee \vrst) \wedge \vsylvh$;
        \item $\vhypo \vee \vlst = \vhypo \vee (\vmst \wedge \vsylvh) = (\vlst \vee \VarS)  \wedge (\vhypo \vee \vmst) = \vsylvh  \wedge (\vhypo \vee \vmst)$;
        \item $\vhypo \vee \vrst = \vhypo \vee (\vmst \wedge \vsylv) = (\vrst \vee \VarS)  \wedge (\vhypo \vee \vmst) = \vsylv \wedge (\vhypo \vee \vmst)$;
        \item $\vhypo \wedge \vmst = (\vhypo \vee \vlst) \wedge (\vmst \wedge \vsylv) = (\vhypo \vee \vrst) \wedge (\vmst \wedge \vsylvh) = \vmst \wedge \VarS$.
    \end{enumerate}
\end{lemma}
\begin{proof}
    The first equality in case \textit{(i)} follows from Theorem~\ref{theorem:vhypo_finite_basis}, Corollary~\ref{cor:VarS_finite_basis} and Proposition~\ref{prop:vhypo_join_vmst_finite_basis}, while the second and third equalities follow from Theorem~\ref{theorem:vsylvh_vsylv_finite_basis} and Proposition~\ref{prop:vhypo_join_vlst_vhypo_join_vrst_finite_basis}.

    The first equalities in cases \textit{(ii)} and \textit{(iii)} follow from Theorem~\ref{theorem:vhypo_identities}, Proposition~\ref{prop:vmst_meet_vsylvh_vmst_meet_vsylv_identities} and Corollary~\ref{cor:vhypo_join_vlst_vhypo_join_vrst_identities}. The second equalities follow from Propositions~\ref{prop:vlst_join_VarS_vrst_join_VarS_finite_basis}, \ref{prop:vhypo_join_vlst_vhypo_join_vrst_finite_basis} and \ref{prop:vhypo_join_vmst_finite_basis}, as \hrOto{}, \hrEto{},\hrOot{} and \hrEot{} are consequences of \hrMt{}. The third equalities follow from Theorem~\ref{theorem:vsylvh_vsylv_finite_basis}.

    The first and second equalities in case \textit{(iv)} follow from Theorem~\ref{theorem:vhypo_finite_basis}, Corollaries~\ref{cor:vmst_finite_basis} and \ref{cor:vmst_meet_vsylvh_vmst_meet_vsylv_finite_basis}, and Proposition~\ref{prop:vhypo_join_vlst_vhypo_join_vrst_finite_basis}, while the third equality follows from Corollary~\ref{cor:vmst_meet_VarS_finite_basis}.
\end{proof}

The correctness of Theorem~\ref{theorem:lattice_2} then follows from Theorem~\ref{theorem:lattice_1} and Lemmas~\ref{lemma:incomparabilities_lattice_2} and \ref{lemma:equalities_lattice_2}.

\subsection{Axiomatic ranks} \label{subsection:L2_axiomatic_ranks}

\begin{lemma} \label{lemma:Mt_not_consequence}
    The identity \hrMt{} is not a consequence of the set of non-trivial identities, satisfied by $\vmst \wedge \VarS$, over an alphabet with two variables.
\end{lemma}
\begin{proof}
    Let $\mathcal{S}$ be the set of all non-trivial identities, satisfied by $\vmst \wedge \VarS$, over an alphabet with two variables. Suppose, in order to obtain a contradiction, that \hrMt{} is a consequence of $\mathcal{S}$. As such, there exists a non-trivial identity $\uord \approx \vord$ in $\mathcal{S}$, and a substitution $\psi$, such that
    \[
    xyxzx = \word_1 \psi(\uord) \word_2,
	\]
	where $\word_1, \word_2$ are words over the two-variable alphabet, and $\psi(\uord) \neq \psi(\vord)$. We can assume, without loss of generality, that $\psi$ does not map any variable to the empty word.
    By Lemma~\ref{lemma:vmst_meet_VarS_shortest_identity}, any proper factor of $xyxzx$ is an isoterm for $\vmst \wedge \VarS$, hence $\word_1$ and $\word_2$ are the empty word, that is, $xyxzx = \psi(\uord)$.
	
	Since $\cont{xyxzx}=\bigl(\begin{smallmatrix} x & y & z \\ 3 & 1 & 1 \end{smallmatrix}\bigr)$ and, by Lemma~\ref{lemma:vmst_meet_VarS_shortest_identity}, there is a lower bound on the length of identities satisfied by $\vmst \wedge \VarS$, we can conclude that, up to the renaming of variables, $x$ occurs exactly thrice and $y$ occurs exactly once in $\uord \approx \vord$. But all factors of length $2$ of $xyxzx$ are distinct and have an occurrence of $x$, hence $\psi(x)$ and $\psi(y)$ must be single variables. As such, the length of $\uord$ is different from that of $xyxzx$, and we obtain a contradiction. Therefore, \hrMt{} is not a consequence of $\mathcal{S}$.
\end{proof}

\begin{corollary} \label{cor:L2_lattice_axiomatic_rank}
    The axiomatic rank of $\vhypo \vee \vlst$ and $\vhypo \vee \vrst$ is 4, and the axiomatic rank of $\vhypo \vee \vmst$ is 3.
\end{corollary}
\begin{proof}
    The first case follows from Lemma~\ref{lemma:L2_Md_R2_not_consequences} and Proposition~\ref{prop:vhypo_join_vlst_vhypo_join_vrst_finite_basis}, and the second case follows from Proposition~\ref{prop:vhypo_join_vmst_finite_basis} and Lemma~\ref{lemma:Mt_not_consequence}.    
\end{proof}

\section{Sublattice of $\mathbb{MON}$ generated by $\vsylvh$, $\vsylv$, $\vlst$, $\vrst$, $\vhypo$ and $\VarMd$.} \label{section:lattice_L3}

\hypertarget{VarMd_definition}{Let $\VarMd$ be the variety that admits a finite equational basis consisting of the identity \hrMd{}.} We now construct the lattice \hrLatt{}, by adding to \hrLatd{} a new generator $\VarMd$.

\begin{figure}[h]
    \centering
    \begin{tikzpicture}
        \draw[very thick] (0,1) -- (-3,0) -- (-6,-1) -- (-6,-3) -- (-6,-6) -- (-6,-11) -- (-3,-12) -- (0,-13);
        \draw[very thick] (0,1) -- (3,0) -- (6,-1) -- (6,-3) -- (6,-6) -- (6,-11) -- (3,-12) -- (0,-13);
        \draw[very thick] (-3,0) -- (0,-1) -- (-2,-3) -- (-6,-6) -- (2,-9) -- (0,-11) -- (-3,-12);
        \draw[very thick] (3,0) -- (0,-1) -- (2,-3) -- (3.6,-6) -- (-2,-9) -- (0,-11) -- (3,-12);
        \draw[very thick] (0,-1) -- (-6,-3) -- (-3.6,-6) -- (-6,-9) -- (0,-11);
        \draw[very thick] (-6,-3) -- (-1.2,-6) -- (-2,-9);
        \draw[very thick] (6,-3) -- (-1.2,-6) -- (2,-9);
        \draw[very thick] (0,-1) -- (6,-3) -- (3.6,-6) -- (6,-9) -- (0,-11);
        \draw[very thick] (-2,-3) -- (1.2,-6) -- (-6,-9);
        \draw[very thick] (-2,-3) -- (6,-6) -- (2,-9);
        \draw[very thick] (2,-3) -- (-3.6,-6) -- (-2,-9);
        \draw[very thick] (2,-3) -- (1.2,-6) -- (6,-9);
        \filldraw[black] (0,1) circle (2pt) node[anchor=south]{$\vbaxt$};
        \filldraw[black] (-3,0) circle (2pt);
        \filldraw[black] (3,0) circle (2pt);
        \filldraw[black] (-6,-1) circle (3pt) node[anchor=east]{$\vsylvh$};
        \filldraw[black] (0,-1) circle (2pt);
        \filldraw[black] (6,-1) circle (3pt) node[anchor=west]{$\vsylv$};
        \filldraw[black] (-6,-3) circle (2pt);
        \filldraw[blue] (-2,-3) circle (3pt) node[anchor=south east]{$\VarMd$};
        \filldraw[red] (2,-3) circle (2pt);
        \filldraw[black] (6,-3) circle (2pt);
        \filldraw[blue] (-6,-6) circle (2pt);
        \filldraw[red] (-3.6,-6) circle (2pt);
        \filldraw[black] (-1.2,-6) circle (2pt) node[above=2pt]{$\VarS$};
        \filldraw[black] (1.2,-6) circle (2pt) node[right=5pt]{$\vmst$};
        \filldraw[red] (3.6,-6) circle (2pt);
        \filldraw[blue] (6,-6) circle (2pt);
        \filldraw[black] (-6,-9) circle (2pt);
        \filldraw[red] (-2,-9) circle (3pt) node[anchor=north east]{$\vhypo$};
        \filldraw[blue] (2,-9) circle (2pt);
        \filldraw[black] (6,-9) circle (2pt);
        \filldraw[black] (-6,-11) circle (3pt) node[anchor=east]{$\vlst$};
        \filldraw[black] (6,-11) circle (3pt) node[anchor=west]{$\vrst$};
        \filldraw[black] (0,-11) circle (2pt);
        \filldraw[black] (-3,-12) circle (2pt);
        \filldraw[black] (3,-12) circle (2pt);
        \filldraw[black] (0,-13) circle (2pt) node[anchor=north]{$\vjst$};
    \end{tikzpicture}
    \caption{Sublattice $\mathbb{L}_3$ of $\mathbb{MON}$ generated by the varieties respectively generated by the \#-sylvester, sylvester, left and right stalactic, and hypoplactic monoids, and the variety $\VarMd$. The points in blue correspond to the elements that are not in \hrLatd{}.}
    \label{fig:lattice_with_sylv_stal_hypo_and_Md}
\end{figure}

\begin{theorem} \label{theorem:lattice_3}
    The Hasse diagram of \hrLatt{} is given in Figure~\ref{fig:lattice_with_sylv_stal_hypo_and_Md}.
\end{theorem}

\subsection{Varietal meets and joins} \label{subsection:L3_meets_joins}

The varieties exclusive to \hrLatt{} have been widely studied, in particular, by Sapir \cite{sapir_finitely_based_monoids}:

\begin{corollary}[{\cite[Corollary~6.6~(i)]{sapir_finitely_based_monoids}}] \label{cor:VarMd_identities}
    The equational theory of $\VarMd$ is the set of balanced identities that satisfy the properties \hlC{}, \hlDo{} and \hlDt{}.
\end{corollary}

\begin{corollary} \label{cor:VarMd_meet_vsylvh_VarMd_meet_vsylv_finite_basis}
    The variety $\VarMd \wedge \vsylvh$ admits a finite equational basis consisting of the identities \hrMd{} and \hrLt{}, and the variety $\VarMd \wedge \vsylv$ admits a finite equational basis consisting of the identities \hrMd{} and \hrRt{}.
\end{corollary}
\begin{proof}
    Follows from the definition of $\VarMd$ and Theorem~\ref{theorem:vsylvh_vsylv_finite_basis}.
\end{proof}

\begin{corollary}[{\cite[Corollary~6.6~(ii)--(iii)]{sapir_finitely_based_monoids}}] \label{cor:VarMd_meet_vsylvh_VarMd_meet_vsylv_identities}
    The equational theory of $\VarMd \wedge \vsylvh$ is the set of balanced identities that satisfy the properties \hlC{} and \hlDo{}, and the equational theory of $\VarMd \wedge \vsylv$ is the set of balanced identities that satisfy the properties \hlC{} and \hlDt{}.
\end{corollary}

\begin{corollary} \label{cor:VarMd_meet_VarS_finite_basis}
    The variety $\VarMd \wedge \VarS$ admits a finite equational basis consisting of the identities \hrMd{}, \hrLt{}, and \hrRt{}.
\end{corollary}
\begin{proof}
    Follows from the definition of $\VarMd$ and Corollary~\ref{cor:VarS_finite_basis}.
\end{proof}

\begin{proposition}[{\cite[Proposition~6.1]{sapir_finitely_based_monoids}}] \label{prop:VarMd_meet_VarS_identities}
    The equational theory of $\VarMd \wedge \VarS$ is the set of balanced identities that satisfy the property \hlC{}.
\end{proposition}

\begin{proposition} \label{prop:VarMd_meet_VarS_varietal_join}
    The variety $\VarMd \wedge \VarS$ is not contained in the join of $\mathbf{COM}$ and any finitely generated variety.
\end{proposition}
\begin{proof}
    Let $F$ be the (relatively) free monoid of rank 2 in $\VarMd \wedge \VarS$, over the alphabet $\{x,y\}$. Suppose, in order to obtain a contradiction, that $F$ is contained in the join of $\mathbf{COM}$ and a finitely generated variety. Then, $F$ is the image of a submonoid $L$ of a direct product of copies of $\bbN$ and copies of a finite monoid under some surjective homomorphism $\phi \colon L \to F$. Let $a,b \in L$ such that $\phi(a) = x$ and $\phi(b) = y$. 
    Then, there exist $p,q \in \bbN$ with $p \neq q$ such that $b^p = b^q$ in the finite part, as such, $b^pab^q$ and $b^qab^p$ are equal in $L$.
    However, their images under $\phi$ are $y^pxy^q$ and $y^qxy^p$, respectively. As the identity $y^pxy^q \approx y^qxy^p$ does not satisfy property \hlC{}, it is not satisfied by the variety $\VarMd \wedge \VarS$. Thus, $y^pxy^q$ and $y^qxy^p$ are not equal in $F$, giving a contradiction.
\end{proof}

As a consequence, it follows that any variety containing $\VarMd \wedge \VarS$ is also not contained in the join of $\mathbf{COM}$ and any finitely generated variety.

\subsection{Proving Theorem~\ref{theorem:lattice_3}} \label{subsection:L3_correctness}

We now prove Theorem~\ref{theorem:lattice_3}, to do this, we require two lemmas that show that lattice \hrLatt{} has no missing varieties, and all the covers are well defined.

\begin{lemma} \label{lemma:incomparabilities_lattice_3}
    The following statements hold:
    \begin{enumerate}[label=(\roman*)]
        \item $\VarMd$ is incomparable with $\vhypo$, $\vhypo \vee \vmst$, $\vsylvh$ and $\vsylv$;
        \item $\VarMd \wedge \vsylvh$ is incomparable with $\vrst$, $\vhypo$, $\vhypo \vee \vmst$ and $\vsylv$;
        \item $\VarMd \wedge \vsylv$ is incomparable with $\vlst$, $\vhypo$, $\vhypo \vee \vmst$ and $\vsylvh$;
        \item $\VarMd \wedge \VarS$ is incomparable with $\vlst$, $\vrst$, $\vhypo$ and $\vhypo \vee \vmst$.
    \end{enumerate}
\end{lemma}
\begin{proof}
    By the definition of $\VarMd$, the identity \hrMd{} is satisfied by $\VarMd$, $\VarMd \wedge \vsylvh$, $\VarMd \wedge \vsylv$, and $\VarMd \wedge \VarS$. On the other hand, by Theorem~\ref{theorem:vhypo_identities}, \hrMd{} is not satisfied by $\vhypo$, hence, it is satisfied by neither $\vhypo \vee \vmst$, $\vsylvh$ nor $\vsylv$. Moreover, by Corollaries~\ref{cor:vlst_vrst_identities} and \ref{cor:VarMd_meet_vsylvh_VarMd_meet_vsylv_identities}, the identity \hrLt{} is satisfied by $\VarMd \wedge \vsylvh$ and $\VarMd \wedge \VarS$ but not by $\vrst$, and \hrRt{} is satisfied by $\VarMd \wedge \vsylv$ and $\VarMd \wedge \VarS$ but not by $\vlst$.
    
    Similarly, by Proposition~\ref{cor:vhypo_join_vmst_identities}, the identity \hrMt{} is satisfied by $\vhypo \vee \vmst$ but, by Proposition~\ref{prop:VarMd_meet_VarS_identities}, is not satisfied by $\VarMd \wedge \VarS$, and therefore is satisfied by neither $\VarMd$, $\VarMd \wedge \vsylvh$ nor $\VarMd \wedge \vsylv$. By Theorem~\ref{theorem:vsylvh_vsylv_identities} and Corollary~\ref{cor:VarMd_meet_vsylvh_VarMd_meet_vsylv_identities}, the identity \hrRt{} is satisfied by $\vsylv$ but neither by $\VarMd \wedge \vsylvh$ nor $\VarMd$ and \hrLt{} is satisfied by $\vsylvh$ but neither by $\VarMd \wedge \vsylv$ nor $\VarMd$.
\end{proof}

\begin{lemma}  \label{lemma:equalities_lattice_3}
    The following equalities hold:
    \begin{enumerate}[label=(\roman*)]
        \item $\VarMd = \vmst \vee (\VarMd \wedge \VarS) = (\VarMd \wedge \vsylvh) \vee \vrst = (\VarMd \wedge \vsylv) \vee \vlst$;
        \item $\VarMd \wedge \vsylvh = \VarMd \wedge (\vlst \vee \VarS) = (\vmst \wedge \vsylvh) \vee (\VarMd \wedge \VarS) = \vlst \vee (\VarMd \wedge \VarS)$;
        \item $\VarMd \wedge \vsylv = \VarMd \wedge (\vrst \vee \VarS) = (\vmst \wedge \vsylv) \vee (\VarMd \wedge \VarS) = \vrst \vee (\VarMd \wedge \VarS)$;
        \item $\VarMd \vee \vhypo = (\VarMd \wedge \vsylvh) \vee (\vhypo \vee \vrst) = (\VarMd \wedge \vsylv) \vee (\vhypo \vee \vlst) = \vmst \vee \VarS$;
        \item $\VarMd \wedge \vhypo = \vmst \wedge \VarS$;
        \item $\VarMd \wedge (\vhypo \vee \vmst) = \vmst$;
        \item $(\VarMd \wedge \VarS) \vee \vhypo = \VarS$.
    \end{enumerate}
\end{lemma}
\begin{proof}
    The first equality in case \textit{(i)} follows from the definition of $\VarMd$, Corollary~\ref{cor:vmst_identities} and Proposition~\ref{prop:VarMd_meet_VarS_identities}, while the second and third equalities follow from Corollaries~\ref{cor:vlst_vrst_identities} and \ref{cor:VarMd_meet_vsylvh_VarMd_meet_vsylv_identities}.

    The first equalities in cases \textit{(ii)} and \textit{(iii)} follow from the definition of $\VarMd$ and Proposition~\ref{prop:vlst_join_VarS_vrst_join_VarS_finite_basis}. The second equalities follow from Propositions~\ref{prop:vmst_meet_vsylvh_vmst_meet_vsylv_identities} and \ref{prop:VarMd_meet_VarS_identities}, while the third equalities follow from Corollary~\ref{cor:vlst_vrst_identities}.

    The first and second equalities in case \textit{(iv)} follow from Corollaries~\ref{cor:vhypo_join_vlst_vhypo_join_vrst_identities} and \ref{cor:VarMd_meet_vsylvh_VarMd_meet_vsylv_identities}, while the third equality follows from Corollary~\ref{cor:vmst_join_VarS_identities}.

    Case \textit{(v)} follows from the definition of $\VarMd$, Theorem~\ref{theorem:vhypo_finite_basis} and Corollary~\ref{cor:vmst_meet_VarS_finite_basis}.

    Case \textit{(vi)} follows from the definition of $\VarMd$, Corollary~\ref{cor:vmst_finite_basis} and Proposition~\ref{prop:vhypo_join_vmst_finite_basis}.

    Case \textit{(vii)} follows from Theorem~\ref{theorem:vhypo_identities} and Propositions~\ref{prop:VarS_identities} and \ref{prop:VarMd_meet_VarS_identities}.
\end{proof}

The correctness of Theorem~\ref{theorem:lattice_3} then follows from Theorem~\ref{theorem:lattice_2} and Lemmas~\ref{lemma:incomparabilities_lattice_3} and \ref{lemma:equalities_lattice_3}.

\subsection{Axiomatic ranks} \label{subsection:L3_axiomatic_ranks}

\begin{corollary} \label{cor:L3_lattice_axiomatic_rank}
    The axiomatic rank of the varieties $\VarMd$, $\VarMd \wedge \vsylvh$, $\VarMd \wedge \vsylv$ and $\VarMd \wedge \VarS$ is 4.
\end{corollary}
\begin{proof}
    The result follows from the definition of $\VarMd$, Lemma~\ref{lemma:L2_Md_R2_not_consequences}, and Corollaries~\ref{cor:VarMd_meet_vsylvh_VarMd_meet_vsylv_finite_basis} and \ref{cor:VarMd_meet_VarS_finite_basis}.    
\end{proof}

\section{Equational theories of hyposylvester and metasylvester monoids} \label{section:hypo_metasylvester}

We now look at two recently introduced plactic-like monoids, the hyposylvester and metasylvester monoids, closely connected to the sylvester monoid, and show that they generate the same variety.

The hyposylvester and metasylvester congruences \cite{novelli_hypo_meta} are generated, respectively, by the relations
\begin{align*}
    \calR_\hs =& \{ (ca\uord b, ac\uord b): a \leq b < c, \uord \in \bbN^* \} \\
    & \cup \{ (b\uord ac, b\uord ca): a < b < c, \uord \in \bbN^* \}, \quad \text{and} \\
    \calR_\ms =& \{ (ab\uord a, ba\uord a): a < b, \uord \in \bbN^* \} \\
    & \cup \{ (b\uord ac\vord b, b\uord ca\vord b): a < b < c, \uord,\vord \in \bbN^* \}. 
\end{align*}

The hyposylvester monoids of countable rank and finite rank $n$ are denoted, respectively, by $\hs$ and $\hsn$, while the metasylvester monoids of countable rank and finite rank $n$ are denoted, respectively, by $\ms$ and $\msn$.

\begin{proposition}
    The hyposylvester monoids of rank greater than or equal to $2$ generates the same variety as the infinite-rank sylvester monoid.
\end{proposition}
\begin{proof}
    Notice that the relations $\calR_\hs$ and $\calR_\sylv$ coincide when restricting them to a two-letter alphabet. Therefore, $\hs_2 = \sylv_2$. On the other hand, since $\calR_\sylv \subseteq \calR_\hs$, we have that $\uord \equiv_\sylv \vord$ implies $\uord \equiv_\hs \vord$, thus $\hs$ is a homomorphic image of $\sylv$. Therefore, we have that
    \[
        \vhs \subseteq \vsylv = \bfV_{\sylv_2} = \bfV_{\hs_2} \subseteq \vhs.
    \]
    Hence, $\vhs = \vsylv$.
\end{proof}

Notice that the previous proof can be generalised to show that any monoid obtained by adding relations to $\calR_\sylv$ that require at least three different letters generates the same variety as $\sylv$.

\begin{proposition}
    The metasylvester monoids of rank greater than or equal to $2$ generates the same variety as the infinite-rank sylvester monoid.
\end{proposition}
\begin{proof}
    On one hand, it is clear that the sylvester congruence contains the metasylvester congruence, hence $\sylv$ is a homomorphic image of $\ms$. As such, all identities satisfied by $\ms$ must be satisfied by $\sylv$. On the other hand, $\ms$ satisfies the identity \hrRt{}: We can use the first defining relation of $\calR_\ms$,
    \[
        ab\uord a \equiv_\ms ba\uord a,
    \]
    for $a < b, \uord \in \bbN^*$, to show that the words obtained from $xyzxty$ and $yxzxty$ by replacing each variable by a word over $\bbN$ are $\equiv_\ms$-congruent. Notice that, if $x$ or $y$ are replaced by the empty word, the evaluations are the same, and if $x$ and $y$ are replaced by non-empty words, then all letters occurring in the evaluations of the prefixes $xy$ and $yx$ also occur in the evaluation of the common suffix $zxty$. Thus, we can reorder the letters in the evaluation of $xy$, using the first defining relation, to obtain the evaluation of $yxzxty$ from the evaluation of $xyzxty$. Therefore, since $\ms$ satisfies an identity that forms a basis for the equational theory of $\sylv$, then all identities satisfied by $\sylv$ must be satisfied by $\ms$. The result follows.
\end{proof}

\section{Concluding remarks} \label{section:conclusions}

The varieties $\VarMd \wedge \vsylvh$ and $\VarMd \wedge \vsylv$ are the only maximal hereditarily finitely based overcommutative varieties of monoids \cite[Theorem~1.60]{lee_book}. In \hrLatt{}, $\vmst$ and $\vhypo$ are the minimal varieties that are not hereditarily finitely based: Consider the varieties $\mathbf{J}_1$ and $\mathbf{J}_2$ generated, respectively, by the Rees factor monoids of $\bbN^*$ over the ideals of all words that are not factors of $xzxyty$, and of $xzytxy$ or $xyzxty$. These varieties are not finitely based, but in fact, they are minimal non-finitely based varieties, or \emph{limit varieties} \cite[Proposition 5.1]{jackson_finiteness_properties}. A consequence of \cite[Lemma~3.3]{jackson_finiteness_properties} is that a variety contains $\mathbf{J}_1$ (resp. $\mathbf{J}_2$) if and only if $xzxyty$ is an isoterm (resp. $xzytxy$ and $xyzxty$ are isoterms) of its equational theory. It is easy to check that $xzxyty$ is an isoterm for $\vhypo$, and $xzytxy$ and $xyzxty$ are isoterms for $\vmst$. 

\begin{corollary}
    $\vmst$ is the minimal non-hereditarily finitely based variety in \hrLato{}, and $\vmst$ and $\vhypo$ are the minimal non-hereditarily finitely based varieties in \hrLatd{} and \hrLatt{}.
\end{corollary}

From the characterizations of equational theories of varieties in \hrLatt{}, collected in the second column of Table~\ref{table:results}, and Remark~\ref{remark:polynomial_time}, we have the following:
\begin{corollary}
    The identity-checking problem of any variety in \hrLatt{} is decidable in polynomial time.
\end{corollary}

Examples of finite bases for the varieties in \hrLatt{} are collected in the third column of Table~\ref{table:results}. The number of variables occurring in each basis is equal to its corresponding axiomatic rank, given in the fourth column.

In the fifth column of Table~\ref{table:results}, we give congruences built from plactic-like congruences, whose corresponding factor monoids generate varieties in \hrLatt{}. Notice that no congruences were obtained for any variety in between, and including, $\VarMd \wedge \VarS$ and $\vmst \vee \VarS$. To obtain such a congruence for a variety $\bfV$, we either take the meet of congruences corresponding to two varieties whose join is $\bfV$, in which case the result is immediate, or we take the join of congruences corresponding to two varieties whose meet is $\bfV$, in which case, we need to verify if the factor monoid indeed generates the variety. For example, as mentioned in Subsection~\ref{subsection:plactic_like_monoids}, the hypoplactic congruence is the join of the \#-sylvester and sylvester congruences, however, $\hypo$ does not generate $\VarS$. If one can find a plactic-like monoid that generates $\VarMd \wedge \VarS$, then one can immediately obtain congruences for the unknown cases.

\begin{question}
    Is there a plactic-like monoid that generates $\VarMd \wedge \VarS$?
\end{question}

As a consequence of Theorem~\ref{theorem:lattice_3} and Proposition~\ref{prop:VarMd_meet_VarS_varietal_join} on one hand, and Corollaries~\ref{cor:vhypo_varietal_join}, \ref{cor:vlst_vrst_varietal_join}, \ref{cor:vmst_varietal_join}, \ref{cor:vjst_varietal_join}, \ref{cor:vlst_meet_vsylv_vrst_meet_vsylvh_varietal_join},  \ref{cor:vmst_meet_VarS_varietal_join}, \ref{cor:vmst_meet_vsylvh_vmst_meet_vsylv_varietal_join}, \ref{cor:vhypo_join_vlst_vhypo_join_vrst_varietal_join} and \ref{cor:vhypo_join_vmst_varietal_join} on the other, we have the following:  

\begin{theorem}
    Let $\bfV$ be a variety in \hrLatt{}. Then, $\bfV$ is the varietal join of $\mathbf{COM}$ and a finitely generated variety if and only if $\bfV$ does not contain the variety $\VarMd \wedge \VarS$, or equivalently, is contained in the variety $\vhypo \vee \vmst$. Otherwise, is not contained in any varietal join of $\mathbf{COM}$ and a finitely generated variety.
\end{theorem}

Exactly half of the varieties in \hrLatt{} are the varietal join of $\mathbf{COM}$ and a finitely generated variety. The sixth column of Table~\ref{table:results} either gives the varietal join, or states that the variety is not contained in such a join.

\begin{landscape}
    
{\renewcommand{\arraystretch}{1.2} 
\begin{longtable}[c]{cccccc}
    \captionsetup{width=.9\linewidth}
    \caption{Results on the varieties in \hrLatt{}. The second column gives the properties that define the equational theories; the fifth gives plactic-like congruences whose corresponding factor monoids generate the varieties; the sixth states if a variety is the join of $\mathbf{COM}$ and a finitely generated variety (giving said variety) or if it is not contained in such a join (`N.c.' stands for `not contained').\label{table:results}}\\
    \toprule
    \textbf{Variety} & \textbf{Eq. properties} & \textbf{Finite basis} & \textbf{Ax. rank} & \textbf{Plactic-like} & \textbf{F.g.v.}\\
    \midrule
    \endfirsthead
    
    \caption{(continued)}\\
    \toprule
    \textbf{Variety} & \textbf{Eq. properties} & \textbf{Finite basis} & \textbf{Ax. rank} & \textbf{Plactic-like} & \textbf{F.g.v.}\\
    \midrule
    \endhead

    $\vbaxt$ & \hyperref[theorem:vbaxt_identities]{(C\textsubscript{pre}),(C\textsubscript{suf})} & \hyperref[theorem:vbaxt_finite_basis]{\begin{tabular}{c} $xzytxyrxsy \approx xzytyxrxsy$, \\ $xzytxyrysx \approx xzytyxrysx$ \end{tabular}} & \hyperref[cor:vbaxt_axiomatic_rank]{6} & \hyperlink{baxt_definition}{$\equiv_\baxt$} & \hyperref[prop:vbaxt_varietal_join]{N.c.} \\ 
    
    \arrayrulecolor{lightgray}\midrule
    
    $\vrst \vee \vsylvh$ & \hyperref[cor:vrst_join_vsylvh_vlst_join_vsylv_identities]{(C\textsubscript{pre}),(S\textsubscript{suf})} & \hyperref[prop:vrst_join_vsylvh_vlst_join_vsylv_finite_basis]{\begin{tabular}{c} $xzytxyry \approx xzytyxry$, \\ $xzytxyrx \approx xzytyxrx$ \end{tabular}} & \hyperref[cor:L1_lattice_axiomatic_rank]{5} & \hyperref[cor:vrst_join_vsylvh_vlst_join_vsylv_generators]{$\equiv_\rst \wedge \equiv_\sylvh$} & \hyperref[prop:VarMd_meet_VarS_varietal_join]{N.c.} \\ 
    
    \midrule
    
    $\vlst \vee \vsylv$ & \hyperref[cor:vrst_join_vsylvh_vlst_join_vsylv_identities]{(S\textsubscript{pre}),(C\textsubscript{suf})} & \hyperref[prop:vrst_join_vsylvh_vlst_join_vsylv_finite_basis]{\begin{tabular}{c} $xzxytxry \approx xzyxtxry$, \\ $xzxytyrx \approx xzyxtyrx$ \end{tabular}} & \hyperref[cor:L1_lattice_axiomatic_rank]{5} & \hyperref[cor:vrst_join_vsylvh_vlst_join_vsylv_generators]{$\equiv_\lst \wedge \equiv_\sylv$} & \hyperref[prop:VarMd_meet_VarS_varietal_join]{N.c.} \\ 
    
    \midrule
    
    $\vsylvh$ & \hyperref[theorem:vsylvh_vsylv_identities]{(C\textsubscript{pre})} & \hyperref[theorem:vsylvh_vsylv_finite_basis]{$xzytxy \approx xzytyx$} & \hyperref[cor:vsylvh_vsylv_axiomatic_rank]{4} & \hyperlink{sylvh_sylv_definition}{$\equiv_\sylvh$} & \hyperref[prop:vsylvh_vsylv_varietal_join]{N.c.} \\ 
    
    \midrule
    
    $\vsylv$ & \hyperref[theorem:vsylvh_vsylv_identities]{(C\textsubscript{suf})} & \hyperref[theorem:vsylvh_vsylv_finite_basis]{$xyzxty \approx yxzxty$} & \hyperref[cor:vsylvh_vsylv_axiomatic_rank]{4} & \hyperlink{sylvh_sylv_definition}{$\equiv_\sylv$} & \hyperref[prop:vsylvh_vsylv_varietal_join]{N.c.} \\ 
    
    \midrule
    
    $\vmst \vee \VarS$ & \hyperref[cor:vmst_join_VarS_identities]{(Sub\textsubscript{2}),(Rst\textsubscript{1,v}),(S\textsubscript{pre}),(S\textsubscript{suf})} & \hyperref[prop:vmst_join_VarS_finite_basis]{\begin{tabular}{c} $xzytxyry \approx xzytyxry$,\\ $xzytxyrx \approx xzytyxrx$,\\ $xzxytxry \approx xzyxtxry$,\\ $xzxytyrx \approx xzyxtyrx$ \end{tabular}} & \hyperref[cor:L1_lattice_axiomatic_rank]{5} & ? & \hyperref[prop:VarMd_meet_VarS_varietal_join]{N.c.} \\ 
    
    \arrayrulecolor{black}\bottomrule
    \newpage
    
    $\vlst \vee \VarS$ & \hyperref[cor:vlst_join_VarS_vrst_join_VarS_identities]{(Sub\textsubscript{2}),(Rst\textsubscript{1,v}),(S\textsubscript{pre})} & \hyperref[prop:vlst_join_VarS_vrst_join_VarS_finite_basis]{\begin{tabular}{c} $xzxytxry \approx xzyxtxry$,\\ $xzxytyrx \approx xzyxtyrx$,\\ $xzytxy \approx xzytyx$ \end{tabular}} & \hyperref[cor:L1_lattice_axiomatic_rank]{5} & ? & \hyperref[prop:VarMd_meet_VarS_varietal_join]{N.c.} \\ 
    
    \arrayrulecolor{lightgray}\midrule
    
    $\vrst \vee \VarS$ & \hyperref[cor:vlst_join_VarS_vrst_join_VarS_identities]{(Sub\textsubscript{2}),(Rst\textsubscript{1,v}),(S\textsubscript{suf})} & \hyperref[prop:vlst_join_VarS_vrst_join_VarS_finite_basis]{\begin{tabular}{c} $xzytxyry \approx xzytyxry$, \\ $xzytxyrx \approx xzytyxrx$, \\ $xyzxty \approx yxzxty$ \end{tabular}} & \hyperref[cor:L1_lattice_axiomatic_rank]{5} & ? & \hyperref[prop:VarMd_meet_VarS_varietal_join]{N.c.} \\
    
    \midrule
    
    $\VarMd$ & \hyperref[cor:VarMd_identities]{(Rst\textsubscript{1,v}),(S\textsubscript{pre}),(S\textsubscript{suf})} & \hyperlink{VarMd_definition}{$xzxyty \approx xzyxty$} & \hyperref[cor:L3_lattice_axiomatic_rank]{4} & ? & \hyperref[prop:VarMd_meet_VarS_varietal_join]{N.c.} \\ 
    
    \midrule
    
    $\vhypo \vee \vmst$ & \hyperref[cor:vhypo_join_vmst_identities]{(Sub\textsubscript{2}),(S\textsubscript{pre}),(S\textsubscript{suf})} & \hyperref[prop:vhypo_join_vmst_finite_basis]{$xyxzx \approx x^2yzx$} & \hyperref[cor:L2_lattice_axiomatic_rank]{3} & \hyperref[cor:vhypo_join_vmst_generators]{$\equiv_\hypo \wedge \equiv_\mst$} & \hyperref[cor:vhypo_join_vmst_varietal_join]{$\mathbf{J}_2 \vee \mathbf{RB}$} \\ 
    
    \midrule
    
    $\VarMd \wedge \vsylvh$ & \hyperref[cor:VarMd_meet_vsylvh_VarMd_meet_vsylv_identities]{(Rst\textsubscript{1,v}),(S\textsubscript{pre})} & \hyperref[cor:VarMd_meet_vsylvh_VarMd_meet_vsylv_finite_basis]{\begin{tabular}{c} $xzxyty \approx xzyxty$, \\ $xzytxy \approx xzytyx$ \end{tabular}} & \hyperref[cor:L3_lattice_axiomatic_rank]{4} & ? & \hyperref[prop:VarMd_meet_VarS_varietal_join]{N.c.} \\ 
    
    \midrule
    
    $\VarMd \wedge \vsylv$ & \hyperref[cor:VarMd_meet_vsylvh_VarMd_meet_vsylv_identities]{(Rst\textsubscript{1,v}),(S\textsubscript{suf})} & \hyperref[cor:VarMd_meet_vsylvh_VarMd_meet_vsylv_finite_basis]{\begin{tabular}{c} $xzxyty \approx xzyxty$, \\ $xyzxty \approx yxzxty$ \end{tabular}} & \hyperref[cor:L3_lattice_axiomatic_rank]{4} & ? & \hyperref[prop:VarMd_meet_VarS_varietal_join]{N.c.} \\ 
    
    \midrule

    $\vhypo \vee \vlst$ & \hyperref[cor:vhypo_join_vlst_vhypo_join_vrst_identities]{(Sub\textsubscript{2}),(S\textsubscript{pre})} & \hyperref[prop:vhypo_join_vlst_vhypo_join_vrst_finite_basis]{\begin{tabular}{c} $xyxzx \approx x^2yzx$, \\ $ xzytxy \approx xzytyx$ \end{tabular}} & \hyperref[cor:L2_lattice_axiomatic_rank]{4} & \hyperref[cor:vhypo_join_vlst_vhypo_join_vrst_generators]{$\equiv_\hypo \wedge \equiv_\lst$} & \hyperref[cor:vhypo_join_vlst_vhypo_join_vrst_varietal_join]{$\mathbf{J}_2 \vee \mathbf{LRB}$} \\ 
    
    \midrule
    
    $\vhypo \vee \vrst$ & \hyperref[cor:vhypo_join_vlst_vhypo_join_vrst_identities]{(Sub\textsubscript{2}),(S\textsubscript{suf})} & \hyperref[prop:vhypo_join_vlst_vhypo_join_vrst_finite_basis]{\begin{tabular}{c} $xyxzx \approx x^2yzx$, \\ $xyzxty \approx yxzxty$ \end{tabular}} & \hyperref[cor:L2_lattice_axiomatic_rank]{4} & \hyperref[cor:vhypo_join_vlst_vhypo_join_vrst_generators]{$\equiv_\hypo \wedge \equiv_\rst$} & \hyperref[cor:vhypo_join_vlst_vhypo_join_vrst_varietal_join]{$\mathbf{J}_2 \vee \mathbf{RRB}$} \\ 
    
    \arrayrulecolor{black}\bottomrule
    
    $\VarS$ & \hyperref[prop:VarS_identities]{(Sub\textsubscript{2}),(Rst\textsubscript{1,v})} & \hyperref[cor:VarS_finite_basis]{\begin{tabular}{c} $xzytxy \approx xzytyx$, \\ $xyzxty \approx yxzxty$\end{tabular}} & \hyperref[cor:L1_lattice_axiomatic_rank]{4} & ? & \hyperref[prop:VarMd_meet_VarS_varietal_join]{N.c.} \\ 
    
    \arrayrulecolor{lightgray}\midrule
    
    $\vmst$ & \hyperref[cor:vmst_identities]{(S\textsubscript{pre}),(S\textsubscript{suf})} & \hyperref[cor:vmst_finite_basis]{\begin{tabular}{c} $xyxzx \approx x^2yzx$, \\ $xzxyty \approx xzyxty$ \end{tabular}} & \hyperref[cor:vmst_axiomatic_rank]{4} & \hyperlink{mst_definition}{$\equiv_\mst$} & \hyperref[cor:vmst_varietal_join]{$\mathbf{RB}$} \\ 
    
    \midrule
    
    $\vmst \wedge \vsylvh$ & \hyperref[prop:vmst_meet_vsylvh_vmst_meet_vsylv_identities]{(S\textsubscript{pre}),(S\textsubscript{1,suf})} & \hyperref[cor:vmst_meet_vsylvh_vmst_meet_vsylv_finite_basis]{\begin{tabular}{c} $xzytxy \approx xzytyx$, \\ $xyxzx \approx x^2yzx$,\\ $xzxyty \approx xzyxty$ \end{tabular}} & \hyperref[cor:L1_lattice_axiomatic_rank]{4} & \hyperref[prop:vmst_meet_vsylvh_vmst_meet_vsylv_generators]{$\equiv_\mst \vee \equiv_\sylvh$} & \hyperref[cor:vmst_meet_vsylvh_vmst_meet_vsylv_varietal_join]{$\mathbf{LRB} \vee \bfV_{J^1}$} \\ 
    
    \midrule
    
    $\vmst \wedge \vsylv$ & \hyperref[prop:vmst_meet_vsylvh_vmst_meet_vsylv_identities]{(S\textsubscript{1,pre}),(S\textsubscript{suf})} & \hyperref[cor:vmst_meet_vsylvh_vmst_meet_vsylv_finite_basis]{\begin{tabular}{c} $xyzxty \approx yxzxty$,\\ $xyxzx \approx x^2yzx$,\\ $xzxyty \approx xzyxty$ \end{tabular}} & \hyperref[cor:L1_lattice_axiomatic_rank]{4} & \hyperref[prop:vmst_meet_vsylvh_vmst_meet_vsylv_generators]{$\equiv_\mst \vee \equiv_\sylv$} & \hyperref[cor:vmst_meet_vsylvh_vmst_meet_vsylv_varietal_join]{$\mathbf{RRB} \vee \bfV_{\overleftarrow{J^1}}$} \\
    
    \midrule
    
    $\vhypo$ & \hyperref[theorem:vhypo_identities]{(Sub\textsubscript{2})} & \hyperref[theorem:vhypo_finite_basis]{\begin{tabular}{c} $xyxzx \approx x^2yzx$,\\ $xzytxy \approx xzytyx$,\\ $xyzxty \approx yxzxty$ \end{tabular}} & \hyperref[cor:vhypo_axiomatic_rank]{4} & \hyperlink{hypo_definition}{$\equiv_\hypo$} & \hyperref[cor:vhypo_varietal_join]{$\mathbf{J}_2$} \\ 
    
    \midrule
    
    $\VarMd \wedge \VarS$ & \hyperref[prop:VarMd_meet_VarS_identities]{(Rst\textsubscript{1,v})} & \hyperref[cor:VarMd_meet_VarS_finite_basis]{\begin{tabular}{c} $xzxyty \approx xzyxty$, \\ $xzytxy \approx xzytyx$, \\ $xyzxty \approx yxzxty$ \end{tabular}} & \hyperref[cor:L3_lattice_axiomatic_rank]{4} & ? & \hyperref[prop:VarMd_meet_VarS_varietal_join]{N.c.} \\
    
    \midrule
    
    $\vlst$ & \hyperref[cor:vlst_vrst_identities]{(S\textsubscript{pre})} & \hyperref[cor:vlst_vrst_finite_basis]{$xyx \approx x^2y$} & \hyperref[cor:vlst_vrst_axiomatic_rank]{2} & \hyperlink{lst_rst_definition}{$\equiv_\lst$} & \hyperref[cor:vlst_vrst_varietal_join]{$\mathbf{LRB}$} \\ 
    
    \arrayrulecolor{black}\bottomrule
    
    $\vrst$ & \hyperref[cor:vlst_vrst_identities]{(S\textsubscript{suf})} & \hyperref[cor:vlst_vrst_finite_basis]{$xyx \approx yx^2$} & \hyperref[cor:vlst_vrst_axiomatic_rank]{2} & \hyperlink{lst_rst_definition}{$\equiv_\rst$} & \hyperref[cor:vlst_vrst_varietal_join]{$\mathbf{RRB}$} \\ 
    
    \arrayrulecolor{lightgray}\midrule
    
    $\vmst \wedge \VarS$ & \hyperref[prop:vmst_meet_VarS_identities]{(S\textsubscript{1,pre}),(S\textsubscript{1,suf})} & \hyperref[cor:vmst_meet_VarS_finite_basis]{\begin{tabular}{c} $xzytxy \approx xzytyx$, \\ $xyzxty \approx yxzxty$, \\ $xyxzx \approx x^2yzx$, \\ $xzxyty \approx xzyxty$ \end{tabular}} & \hyperref[cor:L1_lattice_axiomatic_rank]{4} & \hyperref[prop:vmst_meet_VarS_generators]{$\equiv_\hypo \vee \equiv_\mst$} & \hyperref[cor:vmst_meet_VarS_varietal_join]{$\bfV_{\overleftarrow{J^1} \times J^1}$} \\ 
    
    \midrule
    
    $\vlst \wedge \vsylv$ & \hyperref[prop:vlst_meet_vsylv_vrst_meet_vsylvh_identities]{(S\textsubscript{1,pre})} & \hyperref[prop:vlst_meet_vsylv_vrst_meet_vsylvh_finite_basis]{\begin{tabular}{c} $xyx \approx x^2y$, \\ $x^2y^2 \approx y^2x^2$ \end{tabular}} & \hyperref[cor:L1_lattice_axiomatic_rank]{2} & \hyperref[prop:vlst_meet_vsylv_vrst_meet_vsylvh_generators]{$\equiv_\lst \vee \equiv_\sylv$} & \hyperref[cor:vlst_meet_vsylv_vrst_meet_vsylvh_varietal_join]{$\bfV_{\overleftarrow{J^1}}$} \\ 
    
    \midrule
    
    $\vrst \wedge \vsylvh$ & \hyperref[prop:vlst_meet_vsylv_vrst_meet_vsylvh_identities]{(S\textsubscript{1,suf})} & \hyperref[prop:vlst_meet_vsylv_vrst_meet_vsylvh_finite_basis]{\begin{tabular}{c} $xyx \approx yx^2$, \\ $x^2y^2 \approx y^2x^2$ \end{tabular}} & \hyperref[cor:L1_lattice_axiomatic_rank]{2} & \hyperref[prop:vlst_meet_vsylv_vrst_meet_vsylvh_generators]{$\equiv_\rst \vee \equiv_\sylvh$} & \hyperref[cor:vlst_meet_vsylv_vrst_meet_vsylvh_varietal_join]{$\bfV_{J^1}$} \\ 
    
    \midrule
    
    $\vjst$ & \hyperref[cor:vjst_identities]{(Rst\textsubscript{1})} & \hyperref[cor:vjst_finite_basis]{\begin{tabular}{c} $xyx \approx x^2y$, \\ $xyx \approx yx^2$ \end{tabular}} & \hyperref[cor:vjst_axiomatic_rank]{2} & \hyperlink{jst_definition}{$\equiv_\jst$} & \hyperref[cor:vjst_varietal_join]{$\bfV_{S(\{ab\})}$} \\  
    
    \arrayrulecolor{black}\bottomrule
\end{longtable}
}
\end{landscape}

\section*{Acknowledgements}
The authors thank Alan Cain, António Malheiro, Marianne Johnson and Mark Kambites for their suggestions and helpful comments, and the anonymous referee for their careful reading, providing a reference for Lemma~\ref{lemma:J1_dual_identities} and simpler proofs for Propositions~\ref{prop:vlst_meet_vsylv_vrst_meet_vsylvh_identities}, \ref{prop:vmst_meet_VarS_identities} and \ref{prop:vmst_meet_vsylvh_vmst_meet_vsylv_identities}.

\bibliographystyle{plain}
\bibliography{lvplm.bib}

\end{document}